\documentclass[a4,11pt,reqno]{amsart}

\usepackage{dsfont}
\usepackage{amsxtra}
\usepackage{amsmath}
\usepackage{amscd}
\usepackage{amssymb}
\usepackage{amsfonts}
\usepackage[all]{xy}
\usepackage{mathrsfs}
\usepackage{amsthm} 
\usepackage{color}
\usepackage{upgreek }
\usepackage{tikz}
\usepackage{enumerate}
\usepackage{geometry}
\usepackage{ae}
\usepackage{enumitem}

\numberwithin{itemcounter}{subsection}

\theoremstyle{plain}

\newtheorem{theorem}{Theorem}[section]
\newtheorem{lemma}[theorem]{Lemma}
\newtheorem{definition-lemma}[theorem]{Definition-Lemma}

\newtheorem{proposition}[theorem]{Proposition}

\newtheorem{corollary}[theorem]{Corollary}
\theoremstyle{definition}

\theoremstyle{remark}
\newtheorem{remark}[theorem]{Remark}
\newtheorem{example}[theorem]{Example}
\numberwithin{equation}{section}

\def\bbA{\mathbb{A}}

\def\bbG{\mathbb{G}}

\def\bbN{\mathbb{N}}

\def\frakg{\mathfrak{g}}

\def\frakL{\mathfrak{L}}
\def\frakM{\mathfrak{M}}

\def\calE{\mathcal{E}}

\def\frakL{\mathfrak{L}}

\def\frakg{\mathfrak{g}}

\def\frakgl{\mathfrak{gl}}

\def\bfv{\mathbf{v}}
\def\bfw{\mathbf{w}}

\def\Stab{\operatorname{Stab}\nolimits}

\def\k{{\operatorname{k}\nolimits}}

\def\C{\operatorname{C}\nolimits}

\def\Q{\operatorname{Q}\nolimits}

\def\Z{\operatorname{Z}\nolimits}

\def\rad{\operatorname{rad}\nolimits}

\def\Ker{{\operatorname{Ker}\nolimits}}

\def\Im{{\operatorname{Im}\nolimits}}

\def\Hom{\operatorname{Hom}\nolimits}

\def\Aut{\operatorname{Aut}\nolimits}
\def\Tr{\operatorname{Tr}\nolimits}
\def\End{\operatorname{End}\nolimits}

\def\Ind{\operatorname{Ind}\nolimits}

\def\v{v}
\def\Gr{\operatorname{Gr}\nolimits}

\def\C{\mathbb{C}}

\def\dim{\mathrm{dim}}
\def\Hom{\mathrm{Hom}}

\def\det{\mathrm{det}}

\def\End{\mathrm{End}}

\def\Ind{\mathrm{Ind}}

\def\Z{\mathbb{Z}}

\def\Q{\mathbb{Q}}

\def\qed{\hfill $\sqcap \hskip-6.5pt \sqcup$}
\def\N{\mathbb{N}}

\def\qlb{\overline{\mathbb{Q}}_\ell}

\def\ZZ{\mathbb{Z}}
\def\QQ{\mathbb{Q}}

\def\fqb{\overline{\mathbb{F}}_q}
\def\fq{\mathbb{F}_q}

\def\v{\mathbf{v}}
\def\w{\mathbf{w}}
\def\Exp{\operatorname{Exp}}

\def\kk{{\operatorname{k}\nolimits}}
\def\vol{\operatorname{vol}\nolimits}

\def\Hom{\operatorname{Hom}\nolimits}

\def\End{\operatorname{End}\nolimits}
\def\Aut{\operatorname{Aut}\nolimits}

\def\Ker{\operatorname{Ker}\nolimits}
\def\dim{{\operatorname{dim}\nolimits}}
\def\codim{{\operatorname{codim}\nolimits}}

\numberwithin{equation}{section}

\title[On the number of points of nilpotent quiver varieties over finite fields]{On the number of points of nilpotent quiver \\varieties over finite
fields}
\author{T. Bozec, O. Schiffmann and E. Vasserot}

\begin{document}

\begin{abstract}
We give a closed expression for the number of points over finite fields of the
Lusztig nilpotent variety associated to any quiver, in terms of Kac's $A$-polynomials.
When the quiver has 1-loops or oriented cycles, there are several possible variants of the Lusztig nilpotent variety, and
we provide formulas for the point count of each. This involves nilpotent versions of the Kac $A$-polynomial, which we introduce and for which we give a closed formula similar to Hua's formula for the usual Kac $A$-polynomial.
Finally we compute the number of points over a finite field of the various stratas of the Lusztig nilpotent variety
involved in the geometric realization of the crystal graph. 

\end{abstract}

\maketitle

\setcounter{tocdepth}{3}

\tableofcontents

\setcounter{section}{-1}

\section{Introduction}

\smallskip

The interplay between the geometry of moduli spaces of representations of quivers and the representation theory
of quantum groups has led to numerous constructions and results of fundamental importance for both areas. 
One of the most fundamental object in the theory is the \textit{Lusztig nilpotent variety} introduced in \cite{LusJAMS}, which is a closed substack $\underline{\Lambda}_\mathbf{d}$ of the cotangent stack $T^*\underline{{Rep}}_{\mathbf{d}}(Q)$ of the stack of representations of dimension $\mathbf{d}$ of a quiver $Q$. When $Q$ has no $1$-cycle the stack
$\underline{\Lambda}_{\mathbf{d}}$ is Lagrangian and as shown by Lusztig (resp.~Kashiwara-Saito), its irreducible components are
in one to one bijection with the weight $\mathbf{d}$ piece of the canonical basis (resp.~crystal graph) of $U^+_q(\mathfrak{g}_{Q})$, where $\mathfrak{g}_Q$ is the Kac-Moody Lie algebra associated to $Q$. The stack $\underline{\Lambda}_{\mathbf{d}}$ is singular, and although it can be inductively built by sequences of (stratified) affine fibrations, see \cite{KS}, its geometry remains mysterious. The link
mentioned above with canonical or crystal bases shows that, for quivers without $1$-cycles, the generating series for \textit{top} Borel-Moore
homology groups of $\underline{\Lambda}_{\mathbf{d}}$ is given by
\begin{equation}\label{E:intro1}
\sum_{\mathbf{d}} \dim(H_{top}(\underline{\Lambda}_{\mathbf{d}}, \mathbb{Q}))\,z^\mathbf{d}=
\sum_{\mathbf{d}} \dim(U^+(\mathfrak{g}_Q)[\mathbf{d}])\, z^\mathbf{d} =
\prod_{\alpha \in \Delta^+} (1-z^{\alpha})^{-\dim\; \mathfrak{g}_Q[\alpha]}.
\end{equation}

A natural problem is to extend the formula (\ref{E:intro1}) to the whole cohomology of $\underline{\Lambda}_{\mathbf{d}}$ and to understand its significance from the point of view of representation theory. One first step in this program is worked out in this paper together with its companion \cite{SV}. The aim of the present paper is to carry out the essential step in the computation of the cohomology of 
$\underline{\Lambda}_{\mathbf{d}}$ : we determine the number of points of $\underline{\Lambda}_{\mathbf{d}}$ over finite fields of large
enough characteristic. The answer, given in the form of generating series, is expressed in terms of the Kac polynomials $A_{\mathbf{d}}$ attached to the quiver $Q$. We refer the reader to Theorem~\ref{T:main} for details. In \cite{SV} it is proven that $\underline{\Lambda}_{\mathbf{d}}$ is cohomologically pure, and hence that its Poincar\'e polynomial coincides with its counting polynomial. In addition, the (whole) cohomology of $\underline{\Lambda}_{\mathbf{d}}$ is related there to the Lie algebras introduced by Maulik and Okounkov in \cite{MO}.

\smallskip

Our strategy to compute the number of points of $\underline{\Lambda}_{\mathbf{d}}(\mathbb{F}_q)$ is the following~: we relate the number of points of $\underline{\Lambda}_{\mathbf{d}}(\mathbb{F}_q)$ to the number of points of certain \textit{Lagrangian}
Nakajima quiver varieties $\mathfrak{L}(\mathbf{d}, \mathbf{n})(\mathbb{F}_q)$. Using some purity result for these Nakajima quiver varieties together
with a Poincar\'e duality argument we express the counting polynomial of $\mathfrak{L}(\mathbf{d},\mathbf{n})$ in terms of the
Poincar\'e polynomial of the \textit{symplectic} Nakajima quiver variety $\mathfrak{M}(\mathbf{d},\mathbf{n})$. Finally, we use
Hausel's computation of the Poincar\'e polynomials of $\mathfrak{M}(\mathbf{d},\mathbf{n})$ in terms of Kac polynomials, see \cite{Hausel}.

\smallskip

We note that the generating series for the top homology groups of $\underline{\Lambda}_{\mathbf{d}}$ can be extracted from our formula, and involves only the constant terms of Kac polynomials : combining this with (\ref{E:intro1}) one recovers a proof of Kac's conjecture (first proved in \cite{Hausel}) relating the multiplicities of root spaces in Kac-Moody algebras to the constant term of Kac polynomials.

\smallskip

In the context of \cite{MO} it is essential to allow for \textit{arbitrary} quivers $Q$, such as for instance the quiver with one vertex and $g$ loops. Note that there is no Kac-Moody algebra associated to a quiver which does carry $1$-cycles. In \cite{BozecP}, \cite{Bozec} the first author introduced a quantum group $U_q(\mathfrak{g}_Q)$ attached to an arbitrary quiver $Q$ which coincides with the usual quantized Kac-Moody algebra for a quiver with no $1$-cycles and, 
he generalized to this context several fundamental constructions and results,
in particular the theory of canonical and crystal bases, and an analogue of (\ref{E:intro1}). In the presence of $1$-cycles, Lusztig's nilpotent variety is \textit{not} Lagrangian anymore and one has to consider instead a larger subvariety $\Lambda^{1}$ defined by some \textit{`semi-nilpotency'} condition. In addition, when the quiver $Q$ contains some oriented cycle we consider yet a third subvariety $\Lambda^0$ defined by some weak form of semi-nilpotency. The varieties $\Lambda^0$, $\Lambda^1$
are Lagrangian, contains $\Lambda$ and are in some sense more natural than $\Lambda$ from a geometric perspective. We carry out in parallel the computation of the number of points over finite fields for each of the $\Lambda^0, \Lambda^{1}$ and $\Lambda$. This leads us to introduce two variants $A^0$ and $A^{1}$ of the Kac polynomials, respectively counting nilpotent and $1$-cycle nilpotent indecomposable representations, see Section~1.4 for more details, for which we prove the existence and give an explicit formula similar to Hua's formula. As an application, we
provide a proof of an extension of Kac's conjecture on the multiplicities of Kac-Moody Lie algebras to the setting of arbitrary quivers.

\smallskip

To finish, let us briefly describe the contents of this paper~: the main actors are introduced and our main theorem is stated in Section~1, where several examples are explicitly worked out. Section~2 deals with the existence
of nilpotent Kac polynomials $A^0, A^{1}$, and provides explicit formulas for these in the spirit of Hua's formula for the usual Kac polynomial. In Section~3 we study several subvarieties $\mathfrak{L}^0(\v,\w),$ $ \mathfrak{L}(\v,\w),$ $ \mathfrak{L}^{1}(\v,\w)$ 
of the symplectic Nakajima quiver variety $\mathfrak{M}(\v,\w)$, respectively corresponding to the Lusztig nilpotent varieties 
$\Lambda^0,$ $ \Lambda^{1}$ and $\Lambda$. More precisely, we establish some purity results and compute the counting polynomials of these subvarieties by combining a Poincar\'e duality argument (based on Byalinicki-Birula decompositions) with Hausel's computation of the Betti numbers of $\mathfrak{M}(\v,\w)$. In Section~4 we relate the counting polynomials of 
$\Lambda_{\v}$, $\Lambda^0_\v,$ $ \Lambda^{1}_{\v}$ to the the counting polynomials of the 
$\mathfrak{L}(\v,\w)$, $\mathfrak{L}^0(\v,\w),$ $\mathfrak{L}^{1}(\v,\w)$
and prove our main theorem. Section 5 contains an observation about the counting polynomials of certain natural strata in Lusztig
nilpotent varieties arising in the geometric realization of crystal graphs. Finally, in the appendix we recall the definition of the quantum
group associated in \cite{Bozec} to an arbitrary quiver, give a character formula for it, and use our main Theorem to prove an extension of Kac's conjecture in that context.

\vspace{.2in}

\section{Statement of the result}

\medskip

\subsection{Lusztig nilpotent quiver varieties}\hfill\\

Let $Q=(I,\Omega)$ be a finite\footnote{locally finite quiver would do as well} quiver, with vertex set $I$
and edge set $H$. For $h \in \Omega$ we 
will denote by $h', h''$ the initial and terminal vertex of $h$. Note that we
allow 1-loops, i.e., edges $h$ satisfying $h'=h''$. Set $\v \cdot \v'=\sum_i v_iv'_i$. We denote by
$$\langle \v, \v'\rangle=\v \cdot \v'- \sum_{h\in \Omega} v_{h'}v'_{h''}$$
the Euler form on $\Z^I$, and by $(\bullet,\bullet)$ its symmetrized version  such that
$(\v,\v')=\langle \v, \v' \rangle + \langle \v',\v\rangle$. 
We will call \textit{imaginary} (resp. \textit{real}) a vertex which carries a 1-loop (resp. which doesn't carry an 1-loop) 
and write $I= I^{im} \sqcup I^{re}$ for the associated partition of $I$.

\smallskip

Let $Q^* = (I, \Omega^*)$ be the opposite quiver, in which the direction of every arrow is inverted.
Let $\bar{Q}=(I, \bar \Omega)$ with $\bar \Omega=\Omega \sqcup \Omega^*$ be the doubled quiver, obtained from $Q$ by
replacing each arrow $h$ by a pair of arrows $(h, h^*)$ going in opposite directions.

\smallskip

Fix a field $\k$. For each dimension vector $\v \in \N^I$ we fix an
$I$-graded $\k$-vector space $V=\bigoplus_i V_i$ of graded dimension $\v$ and we set
$$E_{\v}=\bigoplus_{h \in \Omega} \Hom(V_{h'}, V_{h''}), \quad E^*_{\v}=\bigoplus_{h \in \Omega^*} \Hom(V_{h'}, V_{h''}),\quad
\bar{E}_{\v}=\bigoplus_{h \in \bar\Omega} \Hom(V_{h'},
V_{h''}).$$
Elements of $E_{\v}$, $E_{\v}^*$ and $\bar{E}_{\v}$ will be denoted by
$x=(x_h)$, $x^*=(x_{h^*})$ and $\bar x=(x,x^*)$. 

\smallskip

By a flag of $I$-graded vector spaces in $V$ we mean a finite increasing flag of $I$-graded subspaces 
$\big(\{0\}=L^0 \subsetneq L^1 \subsetneq \cdots \subsetneq L^{s}=V\big).$ 
We'll say that $(L^l)$ is a \emph{restricted} flag of $I$-graded vector spaces if 
for all $l$ the vector space $L^l/L^{l-1}$ is concentrated on one vertex.
Since we want to consider arbitrary quivers, we need several notions of nilpotency for quiver representations. We say that :

\begin{itemize}
\item
$x$  is \emph{nilpotent} if there exists a flag of $I$-graded vector spaces $(L^l)$ in $V$ such that
$$x_h (L^l) \subseteq L^{l-1}, \quad l=1, \ldots, s, \quad h \in \Omega,$$

\item
$x$ is  \emph{1-nilpotent} if for any vertex $i\in I^{im}$ there exists a flag of subspaces
$(L_i^l)$ in $V_i$ such that 
$$x_h (L_i^{l}) \subseteq L_i^{l-1}\quad l=1, \ldots, s, \quad h \in \Omega\ \text{with}\ h'=h''=i,$$

\item $\bar x$ is \emph{nilpotent}
 if there exists a flag of $I$-graded vector spaces $(L^l)$ in $V$ such that
$$x_h (L^l) \subseteq L^{l-1}, \quad x_{h^*}(L^l) \subseteq L^{l-1}, \quad l=1, \ldots, s, \quad h \in \Omega.$$

\item $\bar x$ is \emph{semi-nilpotent}\footnote{Our definition of semi-nilpotent representation is not the same as that appearing in \cite{Bozec}; what was called semi-nilpotent in \cite{Bozec} is what we call strongly semi-nilpotent in this paper.}
 if there exists a flag of $I$-graded vector spaces $(L^l)$ in $V$ such that
 $$x_h (L^l) \subseteq L^{l-1}, \quad x_{h^*}(L^l) \subseteq L^l, \quad l=1,
\ldots, s, \quad h \in \Omega,$$

\item $\bar x$  is \emph{strongly semi-nilpotent} if there is a restricted flag of $I$-graded vector spaces $(L^l)$ in $V$ with
$$x_h (L^l) \subseteq L^{l-1}, \quad x_{h^*}(L^l) \subseteq L^{l}, \quad l=1, \ldots, s, \quad h \in \Omega,$$

\item switching the roles of $\Omega$ and $\Omega^*$ we get the notion of \textit{$*$-semi-nilpotent} and \textit{$*$-strongly semi-nilpotent} representation.
\end{itemize}

\smallskip

The sets of nilpotent and 1-nilpotent representations in $E_\v$ form Zariski closed subvarieties $E^0_\v$, $E^{1}_\v$ of $E_{\v}$. 
Likewise, the set of semi-nilpotent, strongly semi-nilpotent, $*$-semi-nilpotent, $*$-strongly semi-nilpotent and nilpotent representations in $\bar E_\v$
form closed subvarieties $\bar{N}_{\v}$, $\bar{N}^{1}_{\v}$, $\bar{N}^{*}_{\v}$, $\bar{N}^{*, 1}_{\v}$ and $\bar{E}^0_{\v}$ of $\bar{E}_{\v}$ respectively.
We have the following inclusions
$$E^0_{\v} \subseteq E^{1}_\v \subseteq E_\v,
\quad
\bar{E}^0_\v \subseteq \bar{N}^{1}_{\v} \subseteq \bar{N}_\v \subseteq \bar{E}_\v, 
\quad \bar{E}^0_\v \subseteq \bar{N}^{*,1}_{\v} \subseteq \bar{N}^*_\v \subseteq \bar{E}_\v.$$ 

\smallskip

\begin{remark} \label{rem:3.1}
(a) For any path $\sigma$ in $Q$,  
let $x_\sigma$ be the composition of $x$ along $\sigma$. 
The representation $x$ of $Q$ is nilpotent if and only if there is an integer $N$ such that $x_\sigma=0$
for each path $\sigma$ of length $\geqslant N$.
Similarly, the representation $\bar x$ of $\bar Q$ is nilpotent if and only if $\bar x_\sigma=0$ for each path $\sigma$ of length 
$\geqslant N$ for some $N$.
Further, $\bar x$ is semi-nilpotent if there exists $N$ such that $\bar x_{\sigma}=0$ for each path
$\sigma$ containing at least $N$ arrows in $Q$, see Proposition~\ref{P:attract}(c).

\smallskip

(b)  Assume that $Q$ has no oriented cycle which does not involve 1-loops, i.e., any oriented cycle in $Q$ is a product of 1-loops.
Then a representation $x$ of $Q$ is nilpotent if and only if it is 1-nilpotent, i.e., we have
$E^0_\v=E^{1}_{\v}$. In addition, the definitions of semi-nilpotent and strongly semi-nilpotent representations coincide, i.e., we have
$\bar{N}^{1}_{\v}=\bar{N}_{\v}$ and likewise for $Q^*$.

\smallskip

(c) Assume now that $Q$ contains no $1$-loop.  Then all representations are $1$-nilpotent, i.e., we have $E^{1}_{\v}=E_{\v}$, and 
any strongly semi-nilpotent representation is automatically nilpotent, i.e., we have $\bar{E}^0_{\v}= \bar{N}^{1}_{\v}=\bar{N}^{*,1}_{\v}$.

\smallskip

(d) Let us finally assume that $Q$ has no oriented cycle at all. In that case, we have $E^0_\v=E^{1}_{\v}=E_{\v}$ and 
all the various semi or $*$-semi-nilpotency conditions
 for elements in $\bar{E}_{\v}$ are equivalent to the nilpotency
condition, i.e., we have $\bar{E}^0_{\v}=\bar{N}_{\v}=\bar{N}^{1}_{\v}=\bar{N}^*_{\v}=\bar{N}^{*,1}_{\v}$.

\smallskip

In the presence of $1$-loops or oriented cycles however the above definitions of
semi-nilpotency and nilpotency do differ and we only have $\bar{E}^0_{\v} \subset \bar{N}_{\v}$ and 
$\bar{E}^0_{\v} \subseteq \bar{N}^*_{\v}$. This is already the case for the Jordan quiver.
\end{remark}

\smallskip

The group $G_{\v}=\prod_i GL(V_i)$ acts on $E_{\v}$,
$\bar{E}_{\v}$ by conjugation. This action preserves the subsets
$\bar{N}_{\v},$ $ \bar{N}^*_{\v}$ and $\bar{E}^0_{\v}$. 

\smallskip

The trace map 
$\Tr: \bar{E}_{\v} \to \k$ such that $\bar x \mapsto \sum_{h\in \Omega}\Tr(x_hx_{h^*})$
identifies $\bar{E}_{\v}$ with $T^*E_{\v}$. We set
$\mathfrak{g}_{\v}=\text{Lie}(G_{\v})=\bigoplus_i \mathfrak{gl}(V_i)$ and identify
$\mathfrak{g}_{\v}$ with its dual $\mathfrak{g}_{\v}^*$ via the trace. 
The moment map for the action of $G_{\v}$ on $\bar{E}_{\v}=T^*E_{\v}$ is the map
$\mu_{\v}: \bar{E}_{\v} \to \mathfrak{g}_{\v}$
such that $\bar x\mapsto \sum_{h \in \Omega} [x_h, x_{h^*}].$

\medskip

Following Lusztig \cite{LusJAMS}, we consider the variety
$$\Lambda_{\v}= \mu_{\v}^{-1}(0) \cap \bar{E}^0_{\v},$$
which is often called the \emph{Lusztig nilpotent variety}. It is a closed subvariety of
$\bar{E}_{\v}$ which, in general, possesses many irreducible components
and is singular. When $Q$ has no $1$-loop $\Lambda_\v$ is Lagrangian, but in general
it is not of pure dimension. Following \cite{Bozec}, we also set
$$\Lambda^{1}_{\v}= \mu_{\v}^{-1}(0) \cap \bar{N}^{1}_{\v},\quad\Lambda^{*,1}_{\v}=\mu_{\v}^{-1}(0) \cap \bar{N}^{*,1}_{\v},$$
which are both Lagrangian subvarieties in $\bar{E}_{\v}$. The above varieties play an important
role in the geometric approach to quantum groups and crystal graphs based on quiver varieties.
We'll call them the \emph{strongly semi-nilpotent varieties}.
See \cite{KS} or
\cite[lect.~4]{SLectures2} and \cite{Bozec} for the case of quiver with 1-loops.
Finally, we set
$$\Lambda^0_{\v}= \mu_{\v}^{-1}(0) \cap \bar{N}_{\v},\quad\Lambda^{*,0}_{\v}=\mu_{\v}^{-1}(0) \cap \bar{N}^{*}_{\v}.$$
These are again Lagrangian subvarieties in $\bar{E}_{\v}$.
These varieties have received less attention than the previous ones because they are not directly linked to quantum groups, but
they are very natural from a geometric point of view.
We'll call them the \emph{semi-nilpotent varieties}.
 We have
$$\Lambda_{\v} \subseteq \Lambda^{1}_{\v} \subseteq \Lambda^0_{\v}, \quad \Lambda_{\v} \subseteq \Lambda^{*,1}_{\v} \subseteq \Lambda^{*,0}_{\v}.$$
When the quiver $Q$ has no $1$-loop, we have $\Lambda_{\v} =\Lambda^{1}_{\v}=\Lambda^{*,1}_{\v}$.
When the
quiver has no oriented cycle besides the $1$-loops we have $\Lambda^{1}_\v=\Lambda^0_{\v}$ and $\Lambda^{*,1}_{\v}=\Lambda^{*,0}_{\v}$.

\medskip

The map $\mu$ and hence the varieties $\Lambda_{\v},$ $\Lambda^{\flat}_{\v},$ $\Lambda^{*,\flat}_\v$ with $\flat=0,1$ are
defined over an arbitrary field $\k$. The aim of this paper is to establish a
formula for the number of points of all these varieties over finite fields, in terms of Kac's
$A$-polynomials \cite{Kac}.
As usual, let $q$ be a power of a prime number $p$.
We will give formulas for the generating series of
$|\Lambda_{\v}(\mathbb{F}_q)|$, $|\Lambda^{1}_{\v}(\mathbb{F}_q)|$ and $|\Lambda^0_{\v}(\mathbb{F}_q)|$. 

\medskip

\subsection{The plethystic exponential}\hfill\\ 

Before we can state our result, we need to fix a few notations. Consider the spaces of power series
$$\mathbf{L}=\Q[[z_i\,;\,i \in I]],\quad
\mathbf{L}_t=\Q(t)[[z_i\,;\,i \in I]].$$
Here $t$ and $z_i$ are formal variables.
For $\v \in \N^I$ we write $z^{\v}=\prod_i z_i^{v_i}$. 
Let 
$$\Exp:\mathbf{L}_t \to \mathbf{L}_t$$ be the plethystic exponential map, which is given by 
$$\Exp(f)=\exp \Big(\sum_{l\geqslant 1}\psi_l(f)/l\Big),$$ 
where $\psi_l: \mathbf{L}_t \to \mathbf{L}_t$ is the $l$th Adams
operator defined by $$\psi_l(z^{\v})=z^{l\v},\quad\psi_l(t^k)=t^{kl}.$$

\medskip

\subsection{The Kac polynomial}\hfill\\

For $\v$ a dimension vector of $Q$ let $A_{\v}(t)$ be the Kac polynomial attached to $Q$ and $\v$. For any finite field
$\mathbb{F}_q$, the integer
$A_{\v}(q)$ is equal to the number $A_\v(\fq)$ of isomorphism classes of absolutely
indecomposable representation of $Q$
of dimension $\v$ over $\mathbb{F}_q$. The existence of $A_{\v}(t)$ is due to
Kac and Stanley, see \cite{Kac}, as is the fact that
$A_{\v}(t) \in \Z[t]$ is unitary of degree $1-\langle \v ,\v \rangle$.
That $A_\v(t)$ has positive coefficients was only recently proved in
\cite{HLV}. By Kac's theorem, we have $A_{\v}(t)=0$ unless $\v$ belongs to the set
$\Delta^+$ of positive roots of $Q$.
Let us consider the formal series in $\mathbf{L}_t$ given by
\begin{equation}\label{E:main}
P_Q(t, z)
= \Exp\Big( \frac{1}{1-t^{-1}}\sum_{\v}{A}_{\v}(t^{-1}) \,z^{\v}\Big).
\end{equation}
Write ${A}_{\v}(t)=\sum_n {a}_{\v,n}\,t^n$. The definition of $P_Q(t)$ may be
rewritten as follows~:
$$P_Q(t,z)=\prod_{\v \in \Delta^+}\;\; 
\prod_{l \geqslant 0}\;\; 
\prod_{n=0}^{1-\langle \v,\v \rangle}(1-t^{-n-l}z^{\v})^{-a_{\v,n}}.$$
Observe that the Fourier modes of $P_Q(t,z)$ are all rational functions in $t$
regular outside of $t=1$.
This allows us to evaluate $P_Q(t,z)$ at any $t\neq 1$.

\medskip

\subsection{The nilpotent Kac polynomials}\label{sec:1.4}\hfill\\ 

To express the point count of $\Lambda^\flat_{\v}$ with $\flat=0,1$ we need some variants of the Kac polynomial. 
For any finite field $\fq$ we let $A^{0}_{\v}(\fq)$ and $A^{1}_{\v}(\fq)$ be the number of absolutely 
indecomposable nilpotent and 1-nilpotent representations of $Q$ of dimension $\v$. By construction, we have
$$A_{\v}(\fq) \geqslant A_{\v}^{1}(\fq) \geqslant A_{\v}^0(\fq).$$
We establish in \S\S 2, 4 the following

\smallskip

\begin{proposition}\label{P:prop1.1} For any quiver $Q$ and any dimension $\v\in \mathbb{N}^I$ there are unique polynomials
$A_{\v}^{\flat}(t)$ in $\Z[t]$ such that for any finite field $\fq$ we have 
$A_{\v}^{\flat}(q)=A_\v^{\flat}(\fq)$. 
Moreover, we have
$A_{\v}^\flat(1)=A_{\v}(1).$
\end{proposition}

\smallskip

We define the formal series in $\mathbf{L}_t$ given by
\begin{equation}\label{E:main2}
\begin{split}
P^{\flat}_Q(t, z)&= \Exp\Big( \frac{1}{1-t^{-1}}\sum_{\v}
{A}^{\flat}_{\v}(t^{-1}) \,z^{\v}\Big).
\end{split}
\end{equation}

\begin{remark}
(a) We conjecture that $A_{\v}^{\flat}(t) \in \mathbb{N}[t]$. 

\smallskip

(b) If $Q$ has no 1-loops then $A_{\v}^{1}(t)=A_{\v}(t)$. 

\smallskip

(c) If any oriented cycle in $Q$ is a product of 1-loops then $A_{\v}^{1}(t)=A_{\v}^0(t)$. 

\end{remark}

\smallskip

\subsection{Kac polynomials and the nilpotent quiver varieties}\hfill\\ 

For an arbitrary quiver $Q$ we consider the generating functions in $\mathbf{L}$ given by
$$\lambda_Q(q,z)=\sum_{\v}
\frac{|\Lambda_{\v}(\mathbb{F}_q)|}{|G_{\v}(\mathbb{F}_q)|}\, q^{\langle \v, \v
\rangle}z^{\v},$$
$$\lambda^{\flat}_Q(q,z)=\sum_{\v}
\frac{|\Lambda^{\flat}_{\v}(\mathbb{F}_q)|}{|G_{\v}(\mathbb{F}_q)|}\, q^{\langle \v, \v
\rangle}z^{\v}.$$
Our main result reads as follows.

\begin{theorem}\label{T:main} 
If the field $\mathbb{F}_q$ is large enough then we have 
\begin{equation}\label{E:theo}
\lambda^{\flat}_Q(q,z)=P^{\flat}_{Q}(q,z), \quad \lambda_Q(q,z)=P_{Q}(q,z).
\end{equation}
\end{theorem}

\smallskip

\begin{remark} (a) By a \textit{large enough field $\mathbb{F}_q$} we mean the following~: for each dimension vector $\v \in \N^I$ there
 exists an integer 
 $N >0$ such that the equality (\ref{E:theo}) holds in degrees $\v' \leqslant \v$ for any field $\mathbb{F}_{q}$ of 
 characteristic $p >N$. 

\smallskip

(b) We define the formal series
$$\zeta_{\v}(q,u)=\Exp \left(A_{\v}(q)\,u\right)=\exp \Big( \sum_{l \geqslant 1} A_{\v}(q^l)\,u^l/l\Big).$$ 
This may be thought
of as the 'zeta function' of the quotient
$$\calE_{\v}(\fq)=\{ x \in E_{\v}(\fq)\,;\, (x_h \otimes\fqb)\;\text{ indecomposable}\}/ G_{\v}.$$
The set $\calE_{\v}(\fq)$ is not the set of $\fq$-points of an algebraic variety.
It is only the set of $\fq$-points of a constructible subset of the moduli stack $E_{\v} /G_{\v}$ 
of representations of $Q$ of dimension $\v$.
So we can not
speak of its zeta function. Using this notation, we may
restate \eqref{E:theo} as the equality
$$\lambda_Q(q,z)=\prod_{\v \in \Delta^+}\prod_{l \geqslant 0}
\zeta_{\v}(q^{-1},q^{-l}z^{\v}).$$
A similar interpretation may be given for $\lambda^{\flat}_Q(z,q)$ in terms of the 'zeta functions' of stacks of $1$-nilpotent, resp. nilpotent, indecomposable representations.

\smallskip

(c) It is possible to give a closed expression for the power series $P_Q(t,z)$ and $P^\flat_Q(t,z)$,
as an immediate consequence of Hua's formula and their nilpotent variants established in \S \ref{sec:2}, see (\ref{E:proof4.5}), (\ref{E:genhua}) and (\ref{E:genhua2}).

\smallskip

(d) The formulas for the point count of $\Lambda^\flat(\v)$ in Theorem~\ref{T:main} also hold in the Grothendieck group
of varieties (or more precisely stacks, see \cite{Bridgeland}) over an algebraically closed field. One simply replaces every occurence of
$q$ by the Lefschetz motive $\mathbb{L}$. This is a consequence of the fact that the point count of the Nakajima quiver varieties
performed in \cite{Hausel} admit a similar motivic lift, see \cite{Wyss}.

 \end{remark}

\medskip

\subsection{Examples}\hfill\\

Let us provide a few simple examples of applications of Theorem~\ref{T:main}.

\smallskip

\subsubsection{Finite Dynkin quivers}
Assume that $Q$ is a finite Dynkin quiver. Then, we have  $A_{\v}(t)=A^{\flat}_{\v}(t)=1$ for all $\v\in
\Delta^+$, and thus
$$\lambda_Q(q,z)=\lambda^\flat_Q(q,z)=\prod_{\v \in \Delta^+} \prod_{l \geqslant 0}
(1-q^{-l}z^{\v})^{-1}=\Exp\Big(\frac{q}{q-1} \sum_{\v \in \Delta^+} z^{\v}
\Big).$$

\smallskip

\subsubsection{Affine quivers} Assume that $Q$ is an affine quiver. Then $\Delta^+=\Delta_{im}^+
\sqcup \Delta^+_{re}$ with
$$\Delta^+_{im}=\N_{\geqslant 1}\delta, \quad
\Delta^+_{re}=\{\Delta^+_0 + \N\delta\} \sqcup \{\Delta_0^- +
\N_{\geqslant 1} \delta\}$$
where $\delta$ is the minimal positive imaginary root, and $\Delta_0$
is the root system of an underlying finite type subquiver $Q_0 \subset Q$. We
have 
$$A_{\v}(t)=A_{\v}^{\flat}(t)=1\ \text{for}\ \v \in \Delta^+_{re}$$
while setting $r=\text{rank}(Q_0)=|I|-1$ we have
$$A_{\v}(t)=A_{\v}^{\flat}(t)=t+r\ \text{for}\ \v \in \Delta^+_{im}$$
when $Q$ is not a cyclic quiver and
$$A_{\v}(t)=A_{\v}^{1}(t)=t+r, \quad A_{\v}^{0}(t)=r+1\ \text{for}\ \v \in \Delta^+_{im}$$
if $Q$ is a cyclic quiver.
This follows from the explicit classification of indecomposable representations
of euclidean quivers, see \cite{CBlectures}.
This yields 
$$\lambda_Q(q,z)=\lambda^{1}_Q(q,z)= \Exp\left( \frac {\sum_{\v \in \Delta_0^+} q\,z^{\v}
  +(1+rq)\,z^{\delta} + \sum_{\v \in \Delta_0^-}
q\,z^{\v+\delta}}{(q-1)(1-z^{\delta})}\right),$$
$$\lambda^0_Q(q,z)=\Exp\left( \frac {\sum_{\v \in \Delta_0^+} q\,z^{\v}
  +(1+r)q\,z^{\delta} + \sum_{\v \in \Delta_0^-}
q\,z^{\v+\delta}}{(q-1)(1-z^{\delta})}\right)$$
for the cyclic quiver of type $A_r^{(1)}$, and
$$\lambda^\flat_Q(q,z)= \lambda_Q(q,z)=\Exp\left( \frac {\sum_{\v \in \Delta_0^+} q\,z^{\v}
  +(1+rq)\,z^{\delta} + \sum_{\v \in \Delta_0^-}
q\,z^{\v+\delta}}{(q-1)(1-z^{\delta})}\right)$$
for all other affine quivers.

\smallskip

\subsubsection{The Jordan quiver}
Assume that $Q$ is the Jordan quiver (with one vertex and one loop). Then
$\Delta^+=\N_{\geqslant 1}$ and we have
$$A_{\v}(t)=t,\quad A_{\v}^{\flat}(t)=1\ \text{for\ all}\ \v \geqslant 1.$$ 
Further $\Lambda^\flat_{\v}$ is the variety of pairs of commuting $\v\times \v$-matrices, with
the second matrix being nilpotent while $\Lambda_{\v}$ is the variety of pairs
of commuting nilpotent matrices.
Theorem~\ref{T:main} gives
\begin{align*}\lambda^\flat_Q(q,z)&=\Exp
\left( \frac{qz}{(q-1)(1-z)}\right)=\prod_{\v\geqslant 1} \prod_{l \geqslant 0} (1-q^{-l}z^\v)^{-1},\\
\lambda_Q(q,z)&=\Exp
\left( \frac{z}{(q-1)(1-z)}\right)=\prod_{\v\geqslant 1} \prod_{l \geqslant 1} (1-q^{-l}z^\v)^{-1}.
\end{align*} 
The above formula should be compared to the Feit-Fine formula for the number of points in the (non-nilpotent) commuting varitieties over finite fields, see \cite{Feit}. Similar formulas also appear in \cite{R-V}.

\medskip

\subsection{Remarks}\label{sec:1.7}\hfill\\ 

We finish this section with several short remarks.

\subsubsection{Poincar\'e duality} It is not difficult to show that
\begin{equation}\label{E:mucount}
\sum_{\v} \frac{|\mu_{\v}^{-1}(0)(\mathbb{F}_q)|}{|G_{\v}(\mathbb{F}_q)|}
\,q^{\langle \v, \v \rangle}\,z^{\v}=\Exp \left(
\frac{q}{q-1} \sum_{\v} A_{\v}(q)\,z^{\v}\right).
\end{equation}
See \cite[thm.~5.1]{Mozgovoy2} for a similar result in the context of
Donaldson-Thomas theory. 
Comparing (\ref{E:mucount})
with (\ref{E:main}) we see that, as far as point counting goes, the quotient stacks
$[\Lambda_{\v}/G_{\v}]$ and $[\mu_{\v}^{-1}(0)/G_{\v}]$ are in some sense
Poincar\'e dual to each other. In fact, our proof of Theorem~\ref{T:main} uses
such a Poincar\'e duality for Nakajima varieties, which are \textit{framed, stable}
versions of $\Lambda_{\v}$ and $\mu_{\v}^{-1}(0)$, see \S 2.1. 

\medskip

\subsubsection{Kac's conjecture} Let $Q$ be a quiver without 1-loops and let
$\mathfrak{g}$ be the associated Kac-Moody 
algebra. By the theorem of Kashiwara-Saito \cite{KS}, conjectured by Lusztig in \cite{Lusconj}, the number of irreducible components of
$\Lambda_{\v}=\Lambda^{1}_{\v}$ is the dimension of the $\v$ weight space $U^+(\mathfrak{g})[\v]$ of the envelopping algebra $U^+(\mathfrak{g})$. 
By the Lang-Weil theorem, we have
$$|\Lambda_{\v}(\mathbb{F}_q)|=| \text{Irr}(\Lambda_{\v})|\, q^{\dim(\Lambda_{\v})} +
O(q^{\dim(\Lambda_{\v})-1/2})$$
from which it follows that
$$ \frac{|\Lambda_{\v}(\mathbb{F}_q)|}{|G_{\v}(\mathbb{F}_q)|} \,q^{\langle \v, \v
\rangle}=|\text{Irr}(\Lambda_{\v})|
+O(q^{-1/2})= \dim( U^+(\mathfrak{g})[\v]) + O(q^{-1/2})$$
which we may write as
\begin{align}\label{E:Kac1}
\begin{split}
\lambda_{Q}(q,z)
&=\sum_{\v} \dim( U^+(\mathfrak{g})[\v]) \,z^{\v} +O(q^{-1/2}),\\
&=\prod_{\v}
(1-z^{\v})^{-\dim(\mathfrak{g}[\v])}+O(q^{-1/2}).
\end{split}
\end{align}
On the other hand, by Theorem~\ref{T:main} we have
\begin{equation}\label{E:Kac2}
\begin{split}
\lambda_{Q}(q,z)&=\Exp\Big( \sum_{\v} a_{\v,0}\,z^{\v} \Big) +O(q^{-1/2})\\
&=\prod_{\v} (1-z^{\v})^{-a_{\v,0}}+ O(q^{-1/2}).
\end{split}
\end{equation}
Combining (\ref{E:Kac1}) and
(\ref{E:Kac2}) we obtain that $a_{\v,0}=\dim(\mathfrak{g}[\v])$, which is
the statement of Kac's conjecture. This conjecture was proved by Hausel in
\cite{Hausel}, using his
computation of the Betti numbers of Nakajima quiver varieties. Note that our
derivation of Theorem~\ref{T:main}
uses Hausel's result in a crucial manner, so that the above is not a new proof
of Kac's conjecture but rather a
reformulation of Hausel's proof in terms of Lusztig nilpotent varieties instead
of Lagrangian Nakajima quiver varieties (more in the spirit of \cite{CBVdB}).

\smallskip

We refer to the appendix for a proof of an extension of Kac's conjecture to the setting of an arbitrary quiver.
This now involves the Lie algebras introduced in \cite{Bozec}.

\medskip

\section{Nilpotent Kac polynomials}\label{sec:2}

\medskip

Given a category $C$, let $\underline C$ be the groupoid formed by the objects of $C$ with their isomorphisms.
If $\underline C$ is finite, we denote by $\vol(\underline C)$ the (orbifold volume)
$$\vol(\underline C)=\sum_{M} \frac{1}{|\text{Aut}(M)|},$$
where the sum ranges over all isomorphism classes of objects  $M$ in $\underline C$ and
$|\text{Aut}(M)|$ is the number of elements in $\text{Aut}(M)$.
Given an algebraic group $G$ acting on a variety $X$, let $X/\!\!/G$ be the categorical quotient and $[X/G]$ be the quotient stack.

\subsection{Stacks of representations and stacks of pairs}\hfill\\

Let us consider an arbitrary quiver $Q$. Denote by $g_i$ be the number of 1-loops at the vertex $i \in I$. 
The aim of this section is to prove the existence of polynomials $A_{\v}^{0}(t), A^{1}_{\v}(t)  \in \Z[t]$ counting nilpotent and
$1$-nilpotent, absolutely 
indecomposable representations of $Q$ of dimension $\v$, 
and to give an explicit formula for them in the spirit of Hua's formula, see \cite{Hua}. 

\smallskip

\subsubsection{From nilpotent endomorphisms to nilpotent Kac polynomials}
Let $C_\kk$ be a Serre subcategory of the category $Rep_\kk$ of all representations of $Q$ over $\kk$.
It is an abelian category of global dimension one. 
Let $\underline{C}_{\v,\kk}$ be the stack of all representations of $Q$ of dimension $\v$ over $\k$ which belong $C_{\kk}$. 
Let $\underline C^{nil}_{\v,\kk}$ be the stack of pairs $(M, \theta)$ with $M\in C_{\v,\kk}$ and $\theta \in \End(M)$ 
a nilpotent endomorphism.

\smallskip

Now, assume that $\kk$ is a finite field. 
For each dimension vector $\v$ we consider the volume of $\underline{C}_{\v,\kk}$ is
$$\vol(\underline C^{nil}_{\v,\,\k})=\sum_{(M,\theta)} \frac{1}{|\text{Aut}(M,\theta)|}.$$
Finally, let $A(C_{\v,\k})$ be the number of absolutely indecomposable representations in $C_{\v,\,\k}$.

\smallskip

\begin{proposition}\label{P:Hua1} The following relation holds in $\mathbf{L}$
\begin{equation}\label{E:prop2.1}
\sum_{\v} \vol(\underline C^{nil}_{\v,\,\fq})\, z^{\v} = \exp\Big(\sum_{\v} \sum_{l \geqslant 1}
\frac{A(C_{\v,\,\mathbb{F}_{q^l}})}{q^l-1}\,z^{l\v}/l\Big).
\end{equation}
\end{proposition}

\smallskip

\begin{proof} The proof follows closely that of \cite[prop.~ 2.2]{SHiggs} and is essentially based on the Krull-Schmitt property of 
$C_{\mathbb{F}_{q}}$ together with the fact that objects in $C_{\mathbb{F}_{q}}$ have connected 
automorphism groups. We briefly sketch the argument for the comfort of the reader and refer to \textit{loc.~cit.}~for details. 

\smallskip

Let $M \in C_{\mathbb{F}_q}$ and let
$M = \bigoplus_i M_i ^{\oplus n_i}$ be a decomposition of $M$ as a direct sum of indecomposables.
Note that the $(M_i,n_i)$ are uniquely determined up to a permutation. 
By the Wedderburns' theorem, the ring $\End(M_i)/\rad(\End(M_i))$ is a field extension of $\mathbb{F}_q$. We set 
$$l_i=[\End(M_i)/\rad(\End(M_i)) : \mathbb{F}_q].$$  
Then, we have 
\begin{align}\begin{split}
\rad({\End}(M))&=\bigoplus_{i \neq j} \Hom(M_i^{\oplus n_i}, M_j^{\oplus n_j}) \oplus \bigoplus_i \rad({\End}(M_i))^{\oplus n_i},\\
\Aut(M)&=\prod_i \Aut(M_i^{\oplus n_i})\; \oplus \;\bigoplus_{i \neq j} \Hom(M_i^{\oplus n_i}, M_j^{\oplus n_j}),\\
\End^0(M)&=\prod_i \End^0(M_i^{\oplus n_i}) \;\oplus \;\bigoplus_{i \neq j} \Hom(M_i^{\oplus n_i}, M_j^{\oplus n_j}).
\end{split}\end{align}
Let $\End^0$ denote the set of nilpotent endomorphisms. We deduce
\begin{equation}\label{E:proofindec1}
\frac{|\End^0(M)|}{|\Aut(M)|}=\prod_i \frac{|\End^0(M_i^{\oplus n_i})|}{|\Aut(M_i^{\oplus n_i})|}=
\prod_i\frac{q^{l_in_i(n_i-1)}}{|GL(n_i, \mathbb{F}_{q^{l_i}})|}.
\end{equation}

Let $Ind(C_{\k})$ be the set of all isomorphism classes of indecomposable objects in $C_{\k}$, 
and let us choose representatives $M_i$ for each $i \in Ind(C_{\k})$. 
Summing (\ref{E:proofindec1}) over all objects $M$ yields an equality
\begin{equation*}
 \begin{split}
  \sum_{\v} \text{vol}(\underline C^{nil}_{\v,\,\fq})\,z^{\v} 
  =\prod_{i \in Ind} \bigg( \sum_{n \geqslant 0} \frac{q^{-l_{i}n}}{(1-q^{-l_{i}n})\cdots (1-q^{-l_{i}})} \,
z^{n\,\dim(M_i)}\bigg).
 \end{split}
\end{equation*}
Applying Heine's formula
$$\sum_{n \geqslant 0} \frac{u^n}{(1-v^n) \cdots (1-v)}=\exp \bigg( \sum_{l \geqslant 1} \frac{u^l}{l(1-v^l)}\bigg)$$
we get
\begin{equation*}
  \sum_{\v} \text{vol}(\underline C^{nil}_{\v,\,\fq})\,z^{\v}=\exp \bigg( \sum_{l \geqslant 1} \sum_{i \in Ind(C_{\fq})} 
  \frac{z^{l\dim(M_i)}}{l(q^{ll_{i}}-1)}\bigg).
\end{equation*}
To prove Proposition~\ref{P:Hua1} it only remains to show the equality
\begin{equation}\label{E:P1}
\sum_{l \geqslant 1} \sum_{i \in Ind(C_{\fq})} \frac{z^{l\dim(M_i)}}{l(q^{ll_{i}}-1)}=\sum_{l \geqslant 1} \sum_{\v} 
\frac{A(C_{\v,\,\mathbb{F}_{q^l}})}{l(q^l-1)}\,z^{l\v}.
\end{equation}
This last equality follows from a standard Galois cohomology argument, see \cite[lem.~2.6]{SHiggs}.
\end{proof}

\smallskip

\subsubsection{The volume of the stack of nilpotent endomorphisms}
Our aim now is to compute the left hand side of (\ref{E:prop2.1}), and to deduce from this both the existence of and a formula 
for $A(C_{\v,\,\k})$. 
For this we introduce a stratification of $\underline C^{nil}_{\v,\,\k}$ 
by Jordan type as follows~: to any object $(M,\theta)$ in $\underline C^{nil}_{\v,\,\k} $ 
with $\theta^s=0$ and $\theta^{s-1} \neq 0$ we associate 
the sequence of surjective maps
$$\xymatrix{ M/ \Im(\theta) \ar@{->>}[r]^-{u_1} & \Im(\theta)/\Im(\theta^2) \ar@{->>}[r]^-{u_2} &\cdots  
\ar@{->>}[r]^-{u_{s-1}}& \Im(\theta^{s-1}) \ar@{->>}[r]^-{u_s} & 0}.$$
Then, we define the Jordan type of $(M,\theta)$ as 
$$J(M,\theta)=(\v_1, \ldots, \v_s),\quad \v_i=\dim(\Ker(u_i)),\quad
\forall i=1, \ldots, s.$$
 Observe that we have $\sum_i i\,\v_i=\v$. We obtain a partition
$$\underline C^{nil}_{\v,\,\k}=\bigsqcup_{\underline{\v}} \underline C^{nil}_{\underline\v,\,\k}$$
where $\underline{\v}$ ranges over all tuples $(\v_1, \ldots, \v_s)$ such that $\sum_i i\, \v_i=\v$ and
$$\underline C^{nil}_{\underline\v,\,\k}=\{(M,\theta) \in\underline C^{nil}_{\v,\,\k}\,;\,J(M,\theta)=\underline{\v}\}.$$

\smallskip

Let $E\to F$ be a map of vector bundles on an algebraic stack.
Then the quotient stack of $E$ by $F$ is an algebraic stack.
We call such a stack a \emph{stack vector bundle}.
We define
\begin{align}\label{o}o({\underline{\v}})=-\sum_i (i-1) \langle \v_i ,\v_i \rangle - \sum_{i<j} i\, (\v_i, \v_j).\end{align}

The following result is proved just as in \cite[prop~3.1]{SHiggs}, see also \cite[\S 5]{SMoz}.

\smallskip

\begin{lemma} The morphism
$\pi_{\underline{\v}}~: \underline C^{nil}_{\underline\v,\,\k}\to \prod_i \underline C_{\v_i\,\k}$
given by
$$\pi_{\underline{\v}}(M,\theta) =  (\Ker(u_1), \Ker(u_2), \ldots, \Ker(u_s))$$
is a stack vector bundle of rank $o({\underline{\v}})$.
\qed
\end{lemma}

\smallskip

We deduce as in \cite[\S 3.1]{GHS} that

\smallskip

\begin{corollary} \label{cor:nil}The volume of $\underline C^{nil}_{\v,\,\fq}$ is given by
$$\vol(\underline C^{nil}_{\v,\,\fq})=\sum_{\underline{\v}} q^{o(\underline{\v})} \prod_i \vol(\underline C_{\v_i,\,\fq}),$$
where the sum ranges over all tuples $\underline{\v}=(\v_1,\v_2,\dots,\v_s)$ with $\sum_i i\, \v_i=\v$.
\qed
\end{corollary}

\smallskip

\subsubsection{The volumes of the stack of (1-)nilpotent representations}
Let $Rep_{\kk}^{0}$ and $Rep_{\kk}^{1}$ be the category of nilpotent and 1-nilpotent representations of $Q$ over some field $\kk$.
By Proposition~\ref{P:Hua1} and Corollary~\ref{cor:nil}, the computation of $A^{0}_{\v}(q)$ and $A^{1}_{\v}(q)$ (for all $\v$ and all $q$) reduces to the computation of the volumes $\vol({\underline{Rep}}^{0}_{\v,\,\fq})$ and $\vol({\underline{Rep}}^{1}_{\v,\fq})$ (for all $\v$ and all $q$). 

\medskip

We begin with $\vol({\underline{Rep}}^{1}_{\v,\fq})$. For integers $n,g\in\N$ we consider the elements $[\infty, n]_t$ and $H(n,g)_t$ in $\Q(t)$ given by
\begin{align*}[\infty, n]_t&=\begin{cases}
{\displaystyle\prod_{k=1}^n (1-t^k)^{-1}}&\ \text{if}\ n>0,\\
1&\ \text{if}\ n=0,\end{cases}\\
H(n,g)_t&=\begin{cases}
{\displaystyle\sum_{(e_k)}t^{(g-1)\sum_{k< l}e_ke_l+g\sum_{k}(e_k)^2} \prod_k \frac{[\infty, ge_{k}-e_{k+1}]_{t}\,[\infty, e_{k+1}]_{t}}{[\infty, ge_{k}]_{t}}}&\ \text{if}\ g>0,\\
[\infty, n_i]_{t} &\ \text{if}\ g=0,\end{cases}\end{align*}
where the sum runs over all tuples $(e_k)$ of positive integers with sum $\sum_ke_k=n$.
More generally, for $\mathbf{n} \in \N^I$ we set
\begin{align*}
[\infty, \mathbf{n}]_t&=\prod_i[\infty, n_i]_t,\\
H(\mathbf{n})_t&=\prod_{i\in I}H(n_i,g_i)_t.
\end{align*}

Recall that $g_i$ is the number of loops at an imaginary vertex $i\in I^{im}$.

\smallskip

\begin{lemma}\label{L:Hua1} For any $\v \in \N^I$ we have
$$\vol({\underline{Rep}}^{1}_{\v,\,\fq})=
q^{-\langle \v,\v\rangle}H(\v)_{q^{-1}}.$$
\end{lemma}

\smallskip

\begin{proof} Set $\k=\fq$.
Let $\mathcal{N}_l^{(g)}$ be set of $g$-tuples of nilpotent matrices in $\End(\kk^l)$, i.e., 
the set of 1-nilpotent representations of dimension $l$ of the quiver with one vertex and $g$ loops.
We have ${\underline{Rep}}^{1}_{\v,\,\k}=[X_{\v}/G_{\v}]$ where
$$X_{\v}=\prod_{\substack{h \in \Omega \\ h' \neq h''}} \Hom(\kk^{v_{h'}}, \kk^{v_{h''}}) \times \prod_{i \in I^{im}} \mathcal{N}^{(g_i)}_{v_i}.$$
We now compute the number of elements $|\mathcal{N}^{(g)}_l|$ in $\mathcal{N}^{(g)}_l$. 
To a point $x=(x_1,\dots,x_g)$ in $\mathcal{N}_l^{(g)}$ we associate
a flag $F=(F_k)$ of subspaces in $\kk^l$ as follows : 
$$F_k= \sum_{h_1, \ldots, h_k=1}^g \Im(x_{h_1} \cdots x_{h_k}).$$
We set $f_k(x)= \dim(F_k)$, and define a partition 
$\mathcal{N}^{(g)}_l=\bigsqcup_{f} \mathcal{N}^{(g)}_{f}$
where
$$\mathcal{N}^{(g)}_{f}=\{x \in \mathcal{N}^{(g)}_l\,;\, f_k(x)=f_k\,,\,\forall k\}.$$
Given a flag $F$ of dimension $f=(f_k)$, we count the number of tuples $x$ whose associated flag is $F$. 
This is a (proper) open subset of $(\mathfrak{n}_F)^g$ where $\mathfrak{n}_F=\{ y \in \End(\kk^l)\,;\, y(F_k) \subseteq F_{k+1}\,,\,\forall k\}$.  
The number of surjective linear maps from $\kk^a$ to $\kk^b$ is equal to 
\begin{equation}\label{E:surjmaps}
|GL_a(\kk)|\,/\,q^{b(a-b)}|GL_{a-b}(\kk)|=q^{ab}\,\frac{[\infty, a-b]_{q^{-1}}}{[\infty, a]_{q^{-1}}}.
\end{equation}
We deduce
that the number of tuples $x$ associated with $F$ is equal to
$$q^{g\sum_{k<h} e_ke_h}\prod_k\frac{[\infty,ge_{k}-e_{k+1}]_{q^{-1}}}{[\infty,ge_{k}]_{q^{-1}}},$$
where $e_k= f_{k-1}-f_k$ for all $k$. Summing over all flags of dimension $f=(f_k)$ gives
\begin{align*}
|\mathcal{N}^{(g)}_{f}|&=
|GL_l(\kk)|\,q^{-\sum_{k\leqslant h}e_ke_h}\,q^{g\sum_{k<h} e_ke_h}\prod_k \frac{[\infty,ge_{k}-e_{k+1}]_{q^{-1}}[\infty, e_{k+1}]_{q^{-1}}}{[\infty,ge_{k}]_{q^{-1}}}\\
&=\frac{q^{(g+1)\sum_{k<h} e_ke_h}}{[\infty,l]_{q^{-1}}}\prod_k\frac{[\infty,ge_{k}-e_{k+1}]_{q^{-1}}[\infty, e_{k+1}]_{q^{-1}}}{[\infty,ge_{k}]_{q^{-1}}}.\end{align*}
Summing  over all possible dimensions $f$ yields $|\mathcal{N}^{(g)}_l|.$
\end{proof}

\medskip

Let us now deal with the similar case of $\vol({\underline{Rep}}^{0}_{\v,\fq})$. For dimension vectors $\v, \w \in \N^I$ we set
$$H(\v,\w,t)=\prod_i \frac{[\infty, \sum_{h''=i} \v_{h'} -\w_i]_t}{[\infty, \sum_{h''=i} \v_{h'}]_t}$$

\begin{lemma}\label{L:Hua15} For any $\v \in \N^I$ we have
$$\vol({\underline{Rep}}^{0}_{\v,\,\fq})=\sum_{\underline{\v}} q^{-\mathbf{a}(\underline{\v})} \prod_l H(\v^{(l)}, \v^{(l+1)}, q^{-1}) \prod_l [\infty,  \v^{(l)}]_{q^{-1}}$$
where
$$\mathbf{a}(\underline{\v})=\sum_{l >k}  \langle \v^{(k)}, \v^{(l)}\rangle + \sum_{i,l} (\v_i^{(l)})^2$$
and where the sum ranges over all tuples $\underline{\v}=(\v^{(1)}, \v^{(2)}, \ldots)$ of elements of $\N^I$ such that $\sum_l \v^{(l)}=\v$.
\end{lemma}

\smallskip

\begin{proof} We stratify the stack ${\underline{Rep}}^0_{\v,\fq}$ as follows. To a nilpotent representation $x$ we associate the
decreasing flag $U(x)$ given by $V=U_0 \supset U_1 \supset U_2 \supset \cdots U_{l(x)}\supset U_{l(x)+1}=\{0\}$ where, for each $i\geqslant 1$, we have
$$U_i=\sum_h x_h(U_{i-1}).$$
We set $\v^{(k)}(x)=\dim (U_{k-1}/U_k) \in \N^I$ and for
$\underline{\v}=(\v^{(1)}, \v^{(2)}, \ldots)$ we set
$${\underline{Rep}}^0_{\v,\fq}[\underline{\v}]=\{x \in {\underline{Rep}}^0_{\v,\fq}\;; \v^{(l)}(x)=\v^{(l)}, \; l \geqslant 1\}$$
so that
$${\underline{Rep}}^0_{\v,\fq}= \bigsqcup_{\underline{\v}}{\underline{Rep}}^0_{\v,\fq}[\underline{\v}]$$
where the sum ranges over all tuples $\underline{\v}=(\v^{(k)})_k$ satisfying $\sum_l \v^{(l)}=\v$.
Next, we compute the volume of each strata. Given a flag $U_{\bullet}$ in $V$ of dimension $\underline{\v}$, the set of nilpotent representations $x$ satisfying $U(x)=U_\bullet$ is isomorphic to
$$\bigoplus_{l> k+1} L(U_k/U_{k+1}, U_{l}/U_{l+1})\oplus \bigoplus_l L^{surj}(U_l/U_{l+1}, U_{l+1}/U_{l+2}).$$
For each $I$-graded vector spaces $V,$ $W$ we set
$$L(V,W)=\bigoplus_h \Hom(V_{h'},W_{h''}), \quad L^{surj}(V,W)=\{(y_h)_h \in L(V,W)\,;\,\forall\;i \in I, \;\Im(\bigoplus_{h''=i} y_h)=W_i\}.$$
The volume of the substack of ${\underline{Rep}}^0_{\v, \fq}$ whose objects are the representations $x$ 
such that $U(x)=U_\bullet$ is thus equal to
$$q^{\sum_{l > k}\sum_{h}\v^{(k)}_{h'}\v^{(l)}_{h''}} \prod_{l}\prod_{i\in I} \frac{[\infty, \sum_{h''=i}\v^{(l)}_{h'}- \v^{(l+1)}_{i}]_{q^{-1}}}{[\infty, \sum_{h''=i}\v^{(l)}_{h'}]_{q^{-1}}}.$$
It only remains to sum up over all flags $U_{\bullet}$ of dimension $\underline{\v}$, and to sum up over all possible tuples $\underline{\v}$.
\end{proof}

\medskip

\subsection{Hua's formulas}\label{sec:2.3}\hfill\\

\subsubsection{Hua's formula}
In this section we write down the analogues of Hua's formula, for our nilpotent versions of Kac polynomials. We begin by recalling the original 
Hua's formula 
(or a formula equivalent to it).
Let $\nu=(\nu^i\,;\,i\in I)$ be an \emph{$I$-partition}, i.e., an $I$-tuple of partitions.
Consider the $I$-tuples of integers $|\nu|$ and $ \nu_k$ in $\bbN^I$ given, for each $k=1,2,\dots$, by
$$|\nu|=( |\nu^i|), \quad \nu_k=(\nu^i_k).$$
Let $X(\nu,t)\in\Q(t)$ be given by
$$X(\nu,t)=\prod_{k} t^{\langle \nu_k, \nu_k\rangle} \,[\infty,\nu_{k}-\nu_{k+1}]_t.$$
We define a power series $r(\w,t,z)$ in $\mathbf{L}_t$ depending on a tuple $\w \in \Z^I$ as 
\begin{align}\label{r}r(\w,t,z)=\sum_{\nu}X(\nu,t^{-1})\,t^{\w \cdot \nu_1}\,z^{|\nu|}.\end{align}
Then Hua's formula \cite{Hua} is
\begin{equation}\label{E:proof4.5}
\text{Exp}\Big(\frac{1}{t-1}\sum_{\v}
A_{\v}(t)\,z^{\v}\Big)=r(0,t,z).
\end{equation}

\smallskip

\subsubsection{The $1$-nilpotent Hua's formula}
Given an $I$-partition $\nu$ as in the previous section, we 
consider the element $X^{1}(\nu,t)\in\Q(t)$ given by
$$X^{1}(\nu,t)=\prod_{k} t^{\langle \nu_k, \nu_k\rangle}\, \,H(\nu_{k}-\nu_{k+1})_t.$$
Next, we define a power series in $r^{1}(\w,t,z)\in\mathbf{L}_t$ depending on a vector $\w \in \Z^I$ as 
\begin{align}\label{r0}r^{1}(\w,t,z)=\sum_{\nu}X^{1}(\nu,t^{-1})\,t^{\w \cdot \nu_1}\,z^{|\nu|}.\end{align}
Note that if $\underline\v=(\v_1,\dots,\v_s)$ is a tuple of elements in $\N^I$ such that $\nu_k-\nu_{k+1}=\v_k$ for all $k$, 
and if the integer $o(\underline\v)$ is as in \eqref{o}, then we have
$$o({\underline{\v}})-\sum_k\langle\v_k,\v_k\rangle=-\sum_k\langle\nu_k,\nu_k\rangle.$$

We can now state the following analogue of Hua's formula, which is a direct consequence of 
Proposition~\ref{P:Hua1}, Corollary~\ref{cor:nil} and Lemma~\ref{L:Hua1}.

\smallskip

\begin{corollary}\label{C:formulaAnil} For any $\v \in \N^I$ there exists a unique polynomial
$A_{\v}^{1}(t) \in \Q[t]$ such that for any finite field $\fq$ we have $A^{1}_{\v}(q)=A^{1}_{\v}( \mathbb{F}_{q })$. 
These polynomials are defined by the following identity~:
\begin{equation}\label{E:genhua}\Exp\Big(\frac{1}{t-1}\sum_\v A^{1}_{\v}(t)\,z^\v\Big)=r^{1}(0,t,z).\end{equation}
\end{corollary}

\smallskip

\subsubsection{The nilpotent Hua's formula}
Given a sequence $\nu^{\bullet}=(\nu^{(1)}, \ldots, )$ of $I$-partition as in the previous section, 
for each $k,l\geqslant 1$ we 
set 
\begin{align*}
|\nu^{\bullet}|=\sum_l |\nu^{(l)}|, \quad \nu_k=\sum_l \nu^{(l)}_k,\quad
\v^{(l)}_k=\nu^{(l)}_k-\nu^{(l)}_{k+1},
\end{align*}
and we consider the element $X^{0}(\nu^{\bullet},t)$ in $\Q(t)$ given by
$$X^{0}(\nu^{\bullet},t)= t^{\mathbf{b}(\nu^{\bullet})} \, \prod_{l,k}\; [\infty, \v_k^{(l)}]_t\;H(\v^{(l)}_{k}, \v^{(l+1)}_{k})_t$$
where
$$\mathbf{b}(\nu^{\bullet})= \sum_k  \Big(\langle \nu_k, \nu_k\rangle-\sum_{l_1 \leqslant l_2 } \langle \v_k^{(l_1)},\v_k^{(l_2)}\rangle + \sum_{i,l} (\v_k^{(l)})^2\Big).$$
Next, we define a power series in $r^{0}(\w,t,z)$ in $\mathbf{L}_t$ depending on a vector $\w \in \Z^I$ as 
\begin{align}\label{r02}r^{0}(\w,t,z)=\sum_{\nu^\bullet}X^{0}(\nu^{\bullet},t^{-1})\,t^{\w \cdot \nu_1}\,z^{|\nu^{\bullet}|}.\end{align}

\smallskip

\begin{corollary}\label{C:formulaAnil2} For any $\v \in \N^I$ there exists a unique polynomial
$A_{\v}^{0}(t) \in \Q[t]$ such that for any finite field $\fq$ we have $A^{0}_{\v}(q)=A^{0}_{\v}( \mathbb{F}_{q })$. 
These polynomials are defined by the following identity~:
\begin{equation}\label{E:genhua2}\Exp\Big(\frac{1}{t-1}\sum_\v A^{0}_{\v}(t)\,z^\v\Big)=r^{0}(0,t,z).\end{equation}
\end{corollary}

\smallskip

\begin{remark}\label{R:kactranspo} Formula \eqref{E:genhua} implies that $A_{\v}^{1}(t)$, like $A_{\v}(t)$, 
does not depend on the orientation of the quiver $Q$. This is not the case for $A^0_{\v}(t)$ in general. However, since
indecomposability is obviously preserved under transposition, the polynomial $A^0_{\v}(t)$ is invariant under changing the direction of 
\textit{all} arrows.

\end{remark}

\medskip

\subsection{Proof of Proposition~\ref{P:prop1.1}}\hfill\\

To finish the proof of Proposition~\ref{P:prop1.1} we must prove that for each $\flat=0,1$ we have
$$A^\flat_{\v}(t) \in \Z[t],\quad A_{\v}^\flat(1)=A_{\v}(1).$$ 

\subsubsection{The integrality}
We will use the following argument 
due to Katz, see \cite[\S 6]{HRV}.

\smallskip

\begin{lemma}[Katz]\label{L:Katz} Let $Z$ be a constructible set defined over $\mathbb{F}_q$. 
Assume that there exists a polynomial $P(t)$ in $\mathbb{C}[t]$ such that for any
integer $l\geqslant 1$ we have $|Z(\mathbb{F}_{q^l})|=P(q^l)$. Then $P(t) \in \Z[t]$.  
\qed
\end{lemma}

\smallskip

In order to use the above result, we must relate the polynomials $A^0_{\v}(t)$ and $A^{1}_{\v}(t)$ to the point count of some constructible sets. 
Set $\k=\fq$.
The sets of nilpotent and 1-nilpotent representations in $E_\v$ is the set of $\fq$-points of the Zariski closed subvarieties $E^0_{\v}$ and $E^{1}_{\v}$.
For $\flat=0$ or 1 we put
\begin{align*}
\text{Aut}^{\flat}_{\v}(\fq)&=\{(x,g)\in E^{\flat}_{\v}(\fq) \times GL_{\v}(\fq)\,;\,g \in \Aut_{Rep_{\fq}}(x)\},\\
\text{Aut}^{\flat,\,a.i.}_{\v}(\fq)&=\{(x,g)  \in \text{Aut}^\flat_{\v}(\fq)\,;\,x\; \text{is\;absolutely\ indecomposable}\}.
\end{align*}
Then $\text{Aut}^\flat_{\v}(\fq)$ is the set of $\fq$-points of an algebraic variety defined over $\mathbb{F}_q,$ 
while $\text{Aut}^{\flat, \,a.i.}_{\v}(\fq)$ is a constructible subset of $\text{Aut}_{\v}(\fq)$. 
The set $\text{Aut}^{\flat, a.i.}_{\v}(\fq)$ is compatible with field
extensions, i.e., we have
$$\text{Aut}^{\flat,\, a.i.}_{\v} (\mathbb{F}_{q^l})=\{(x, g)\in E^{\flat}_{\v}(\mathbb{F}_{q^l}) \times GL_{\v}(\mathbb{F}_{q^l})\,;\,x \;\text{is\;absolutely\; indecomposable}\,,\, g \in \Aut_{Rep_{\mathbb{F}_{q^l}}}(x)\}.$$
This property is not true if one replaces \emph{absolutely indecomposable} by \emph{indecomposable}.
Hence, for any integer $l \geqslant 1$, we have
$$\frac{|\text{Aut}^{\flat,\, a.i.}_{\v}(\mathbb{F}_{q^l})|}{|GL_{\v}(\mathbb{F}_{q^l})|}=|\{x \in E^\flat_{\v}(\mathbb{F}_{q^l})\,;\, x\;\text{is\;absolutely\; indecomposable}\}/\sim|=A^{\flat}_{\v}(q^l).$$
By Lemma~\ref{L:Katz} we deduce that 
$$A^{\flat}_{\v}(t) \cdot \prod_{i}\prod_{k=0}^{v_i-1} (t^{v_i}-t^k) \in \Z[t].$$
From the facts that $A^{\flat}_{\v}(t) \in \mathbb{Q}[t]$ and $ \prod_{i}\prod_{k=0}^{v_i-1} (t^{v_i}-t^k)$ is monic
we deduce that $A^{\flat}_{\v}(t) \in \Z[t]$ as wanted. 

\medskip

\subsubsection{The value at 1}

Let us now prove the equality $A_{\v}^{1}(1)=A_{\v}(1)$. 
Denote by $\Ind_{\v}(\fq)$ and $\Ind^{1}_{\v}(\fq)$ the set of isomorphism classes of absolutely indecomposable and 1-nilpotent absolutely indecomposable
representations of $Q$ of dimension $\v$ over $\fq$. Let the finite group $\fq^\times$ act on $\Ind_{\v}(\fq)$ as follows :
$$u \cdot x=(u^{\epsilon(h)}x_h)_h\ \text{with}\  \epsilon(h)=\begin{cases} 1 & \text{if}\;h' \neq h''\\ u & \text{if}\; h'=h.\end{cases}$$ 

\smallskip

Let us consider the possible sizes of the $\fq^\times$-orbits in $\Ind_{\v}(\fq)$. 
If $x \in \Ind_{\v}(\fq)$ and $u \in \fq^\times$ are such that $u \cdot x \simeq x$ then there is an element $g=(g_i)\in GL_{\v}(\fq)$ such 
that we have $u\cdot xg= gx$. In particular, if $h$ is a 1-loop at the vertex $i$ then $x_h$ maps the generalized eigenspace
$V_{i,\lambda}$ of $g_i$ to $V_{i, u\lambda}$. We deduce that if $x_h$ is not nilpotent then
there is an integer $l\in(0,v_i]$ such that $u^l=1$.
More generally, if there is a non-nilpotent polynomial $P(x_{h_1}, \ldots, x_{h_{g_i}})$ in the $1$-loops $h_1, \ldots, h_{g_i}$ at $i$ then there is an integer $l\in(0,v_i]$ 
such that $u^{l deg(P)}=1$. 

\smallskip

\begin{lemma}\label{L:engel} Let $k$ be a field, $V$ a finite-dimensional $k$-vector space of dimension $d$ and $\{x_1, \ldots, x_n\} \subset \End(V)$. If all the polynomials $P(x_1, \ldots, x_n)$ with no constant terms of degree at most $d^2$ are nilpotent then there exists a flag of subspaces $W_{\bullet}$ in $V$ such that
$x_i(W_j) \subseteq W_{j-1}$ for all $i,j$.
\end{lemma}

\smallskip

\begin{proof}
By Engel's theorem, it is enough to prove that the Lie subalgebra of $gl(V)$ generated by $\{x_1, \ldots, x_n\}$ consists of nilpotent
elements. Set $U =\sum_i k x_i$. Consider the increasing filtration $U=U_0 \subseteq U_1 \subseteq \cdots$, where
$U_i= U_{i-1} + [U,U_{i-1}]$. This filtration stabilizes at, say, $U_k$ to the Lie subalgebra generated by $\{x_1, \ldots, x_n\}$.
Moreover $U_l \neq U_{l+1}$ for all $l<k$. Hence $k \leqslant d^2$. It remains to notice that any element of $U_k$ is a polynomial in $x_1, \ldots, x_n$ of degree at most $k$.
\end{proof}

\smallskip

We deduce from Lemma~\ref{L:engel} that an element
$x \in \Ind_{\v}(\fq)$ is not 1-nilpotent if and only if there exists some vertex $i$ and some polynomial $P(x_{h_1}, \ldots, x_{h_{g_i}})$ 
with zero constant term of degree $t \leqslant \v_i^2$ and consisting of $1$-loops at $i$, which is not nilpotent.
Set $N=\sup\{v_i^3\,;\,i\in I\}$. By the above, any $u \in\fq^\times$ stabilizing a non $1$-nilpotent element $x \in  \Ind_{\v}(\fq)$, 
satisfies $u^k=1$ for some integer $k\in(0,N]$. Hence for any $x \in  \Ind_{\v}(\fq) \backslash \Ind^{0}_{\v}(\fq)$ we have
$$|\Stab_{\fq^\times} (x)| \leqslant \sum_{k=1}^{N} |\mu_k(\fq)| \leqslant \binom{N}{2}.$$
It follows that for such $x$ we have
$$|\fq^\times \cdot x| \in \frac{q-1}{\binom{N}{2}!}\N.$$

\smallskip

Thus, for any field $\fq$ we have
\begin{equation}\label{E:qequalsone}
(q-1)^{-1}(A_{\v}(q)-A_{\v}^{1}(q))=(q-1)^{-1}\cdot |\Ind_{\v}(\fq) \backslash \Ind^{0}_{\v}(\fq)| \in \frac{1}{\binom{N}{2}!}\N.
\end{equation}
This implies that the polynomial $A_{\v}(t)-A_{\v}^{1}(t)$ is divisible by $(t-1)$. Otherwise, we would have
$$\frac{1}{t-1}(A_{\v}(t)-A_{\v}^{1}(t)) \in \Z[t] + \frac{c}{t-1}$$
for some nonzero $c \in \Z$ and (\ref{E:qequalsone}) would be false for $q \gg 1$.

\medskip

The proof that $A^0_{\v}(1)=A_{\v}(1)$ follows the same lines. Denote by $\Ind^0_{\v}(\fq)$ the set of all absolutely indecomposable nilpotent representations of $Q$ of dimension $\v$. We now use the action of $\fq^\times$ on $\Ind_v(\fq)$ and $\Ind^0_{\v}(\fq)$ defined by 
$$u \cdot x= u \cdot (x_h)= (ux_h).$$
Arguing as above, we see that the stabilizer of any representation $x \in \Ind_{\v}(\fq)\backslash \Ind_{\v}^0(\fq)$ consists of elements $u \in \fq^\times$ satisfying $u^k=1$ for some $k \in (0,N']$, where $N'$ is an integer depending on $\v$ but not on $q$.
Specifically, one can take $N'=(\sum_i \v_i)^3$. The rest of the argument is the same.

\smallskip

\begin{example} Let us consider the case of the quiver with one vertex and $g$ loops. Then
\begin{align*}
A_{1}^{1}(t)&=1, \\
A_{2}^{1}(t)&=\frac{t^{g}-1}{t-1},\\
A_3^{1}(t)&= t^{2(g-1)} + \frac{t^{2(g-1)}-1}{t^2-1} \cdot \left( \frac{t^{g+1}-1}{t-1} + \frac{t^g-1}{t-1}\right)
\end{align*}
whereas
\begin{align*}
A_1(t)&=t^g, \\
A_2(t)&=t^{2g-1}\frac{t^{2g}-1}{t^2-1},\\
A_3(t)&=\frac{t^{9g-3}-t^{5g+1}-t^{5g}-t^{5g-1}+t^{3g-1}+t^{3g-2}}{(t^2-1)(t^3-1)}.
\end{align*}
Further, we have $A_1(1)=A_1^{1}(1)=1,$ $ A_2(1)=A_2^{1}(1)=g$ and $A_3(1)=A_3^{1}(1)=2g^2-g$.

\end{example}

\medskip

\section{Nakajima's quiver varieties and their attracting subvarieties}

\medskip

In this section we consider variants of $\Lambda_{\v}$ and $\Lambda^{\flat}_{\v}$ in the context of Nakajima quiver varieties, and compute their number of points over finite fields. This will then 
be used in the next section to prove Theorem~\ref{T:main}. Let $\kk$ be any field.
When we want to specify the field over which we consider a variety $X$ we write $X/\kk.$

\smallskip

\subsection{Quiver varieties}\hfill\\

Let $Q$ be a finite quiver.  We recall the definition of Nakajima quiver
varieties and state some of their properties.
See \cite{Nakajima} for details, and see \cite[\S 4]{Mozgovoy} in the case of an arbitrary field.
Fix $\v, \w \in \N^I$. Let $V, $ $W$ be $I$-graded $\kk$-vector spaces of dimensions $\v,$ $ \w$ and set
$$M(\v,\w) =\bigoplus_{h \in \bar\Omega} \text{Hom}(V_{h'}, V_{h''}) \oplus
\bigoplus_{i \in I} \text{Hom}(V_i, W_i) \oplus \bigoplus_{i \in I}
\text{Hom}(W_i, V_i).$$
Elements of $M(\v,\w)$ will be denoted $m=(\bar x, p,q)$. The space
$M(\v,\w)$ carries a natural symplectic structure, and the group $G_{\v}=\prod_i
GL(V_i)$ acts in a Hamiltonian
fashion. The moment map is
\begin{equation}\label{E:momentmap}
\mu~: M(\v,\w)  \longrightarrow \bigoplus_i \frakgl(V_i),\quad
(\bar x, p, q)  \mapsto \Big( \sum_{\substack{h
\in \Omega \\ h'=i}} x_{h^*}x_h - \sum_{\substack{h \in \Omega\\ h''=i}} x_{h}x_{h^*} +
q_ip_i\,;\,i\in I\Big)
\end{equation}

\smallskip

The categorical quotient of $\mu^{-1}(0)$ by $G_\v$ is the $\kk$-variety
$$\frakM_0(\v,\w)=\mu^{-1}(0) /\hspace{-.05in}/ G_{\v}=\text{Spec}\Big(\bigoplus_{n\in\bbN}\k[\mu^{-1}(0)]^{G_\v}\Bigr).$$ 
It is affine and singular in general. Let $[\bar x,p,q]$ denote the image in $\frakM_0(\v,\w)$ of the triple $m=(\bar x,p,q)$ in $\mu^{-1}(0)$.
We'll abbreviate $0=[0,0,0]$.
If the field $\kk$ is algebraically closed then the set of $\kk$-points of $\frakM_0(\v,\w)$ 
is to the set of closed $G_\v(\kk)$-orbits in $\mu^{-1}(0)(\kk)$
so that $[\bar x,p,q]$ is the unique closed orbit in the closure of the $G_\v(\kk)$-orbit of $(\bar x,p,q)$.

\smallskip

Let $\theta$ be the character of $G_\v$ defined by $\theta(g_i)=\prod_{i\in I}\det(g_i)^{-1}$.
Consider the smooth symplectic quasiprojective variety of dimension
$2d(\v,\w)=2\v \cdot \w - (\v,\v)$ given by
$$\frakM(\v,\w)=\text{Proj}\Big(\bigoplus_{n\in\bbN}\k[\mu^{-1}(0)]^{\theta^n}\Bigr),$$
where
$$\k[\mu^{-1}(0)]^{\theta^n}=\{f\in\k[\mu^{-1}(0)]\,;\,f(g\cdot m)=\theta^n(g)f(m)\,,\,\forall m\in\mu^{-1}(0)\}.$$
There is a natural projective morphism 
$$\pi : \frakM(\v,\w) \to \frakM_0(\v,\w).$$
We say that an element
$(\bar x, p, q) \in \mu^{-1}(0)(\kk)$ is semistable if
the following condition is satisfied~:
$$\big(V' \subset \bigoplus_i \text{Ker}(p_i)\;\;\text{and}\;\; \bar x(V')
\subset V' \; \big)\Rightarrow V'=\{0\}.$$
Let $\mu^{-1}(0)^s$ be the open subset of $\mu^{-1}(0)$ consisting of semistable points. 
Using the Hilbert-Mumford criterion one finds that if $\kk$ is algebraically closed then
$\frakM(\v,\w)$ is the geometric quotient $$\frakM(\v,\w)=\mu^{-1}(0)^s/\!\!/G_{\v}$$
and the map $\pi$ takes the $G_\v$-orbit of $(\bar x,p,q)$ to $[\bar x,p,q]$.

\smallskip

Consider the following closed subvarieties of $\frakM(\v,\w)$~:
\begin{align*}
\frakL^0(\v,\w)&=\{G_{\v}\cdot (\bar x,p,q)\,;\, q=0, \,\bar x\;\text{is\;semi-nilpotent}\},\\
\frakL^{1}(\v,\w)&=\{ G_{\v}\cdot (\bar x, p, q)\,;\, q=0,\,\bar x\;\text{is\;strongly\;semi-nilpotent}\},\\
\frakL(\v,\w)&=\{ G_{\v}\cdot (\bar x, p, q) \,;\, q=0,\;\bar x\;\text{is\;nilpotent}\},\\
\frakM^0(\v,\w)&=\{G_{\v}\cdot (\bar x, p,q) \in \frakM(\v,\w)\,;\, x\;\text{is\;nilpotent}\},\\
\frakM^{1}(\v,\w)&=\{G_{\v}\cdot (\bar x, p,q) \in \frakM(\v,\w)\,;\, x\;\text{is\;$1$-nilpotent}\}.
\end{align*}
By definition, we have
$$\frakL^0(\v,\w)=\Big[ \mu^{-1}(0)^s \cap \Big(\Lambda^0_{\v} \times \bigoplus_i \text{Hom}(V_i,W_i)\Big)\Big] \Big/ G_{\v},$$
and a similar description $\frakL^{1}(\v,\w),$ $ \frakL(\v,\w)$ and $\frakM^\flat(\v,\w)$. 
Note that the definitions of $\frakM^{0}(\v,\w)$ and $\frakM^{1}(\v,\w)$ only involve the edges in $\Omega$.
The variety $\frakL(\v,\w)=\pi^{-1}(0)$ is projective.
If $Q$ has no 1-loop then we have
$$\frakL(\v,\w)=\frakL^{1}(\v,\w),\quad\frakM(\v,\w)=\frakM^{1}(\v,\w).$$
If $Q$ has no oriented cycle besides products of $1$-loops then we have 
$$\frakL^{1}(\v,\w)=\frakL^0(\v,\w),\quad\frakM^0(\v,\w)=\frakM^{1}(\v,\w).$$
In general the following inclusions are strict $$\frakL(\v,\w) \subseteq  \frakL^{1}(\v,\w) \subseteq \frakL^0(\v,\w).$$ 
The variety $\frakL(\v,\w)$ may not be equidimensional.
The variety $\frakL^{1}(\v,\w)$ is Lagrangian in $\frakM(\v,\w)$ by \cite[thm~1.15]{Bozec}.
The same holds for $\frakL^0(\v,\w)$. Their irreducible components may be parametrized by
connected components of the variety of fixed points in $\frakM(\v,\w)$ under suitable torus actions, see Proposition~\ref{P:lag} below.

\smallskip

Replacing $Q$ by its opposite $Q^*$ we define the subvarieties $\frakL^{*,0}(\v,\w)$ and $\frakM^{*,0}(\v,\w)$ of  $\frakM^*(\v,\w)$.  
Note that $\frakM^*(\v,\w)$ and $\frakM(\v,\w)$ are canonically isomorphic. Under this isomorphism we have 
$\frakL^0(\v,\w)=\frakL^{*,0}(\v,\w)$. We have the following chain of inclusions
$$
\xymatrix{
& \frakL^{1}(\v,\w) \ar[r] & \frakL^0(\v,\w) \ar[r] & \frakM^{0}(\v,\w) \ar[r] & \frakM^{1}(\v,\w) \ar[rd]&\\
\frakL(\v,\w) \ar[rd] \ar[ur]& & & & &\frakM(\v,\w)\\
& \frakL^{*,1}(\v,\w) \ar[r] & \frakL^{*,0}(\v,\w) \ar[r] & \frakM^{*,0}(\v,\w) \ar[r] & \frakM^{*,1}(\v,\w) \ar[ru]&
}$$

\medskip

\subsection{Reminder on tori actions on varieties}\hfill\\

\subsubsection{The Bialynicki-Birula decomposition}

An \emph{affine space bundle} of rank $d$ over $X$ is a map $f:Y\to X$ such that $X$ can be covered by open affine subsets 
$U_i$ such that
$f^{-1}(U_i)\simeq\bbA^d\times U_i$ and $f$ corresponds under this bijection to the projection on the second factor.

\smallskip

The proof of Bialynicki-Birula uses the assumption that the field $\kk$ is algebraically closed (of any characteristic)
and that the variety $X$ is smooth and projective. We'll need it for any field $\kk$ and for non smooth quasi-projective varieties $X$.
The restriction on the field is not essential, since given a partition
$X=\bigsqcup_{\rho}X_{\rho}$ of a $\fqb$-variety $X$ into locally closed subsets with affine space bundles 
$p_{\rho}~: X_{\rho} \to F_\rho$
we can always assume that the decomposition and the maps $p_\rho$ are defined over $\fq$ up to taking the field $\fq$ large 
enough to contain
the field of definition of $X$, $X_\rho$, $F_\rho$ and $p_\rho$ for all parameter $\rho$.

\smallskip

In the general situation, it is possible to use Hesselink's generalization of the Bialynicki-Birula decomposition 
\cite[thm~5.8]{Hesselink}
which holds true for arbitrary fields $\kk$ and for non smooth quasi-projective varieties $X$.
In particular, the following holds. Let $X$ be an $\fq$-variety with a $\bbG_m$-action. Assume that $X$ embeds equivariantly in a 
projective space with a diagonalizable
$\bbG_m$-action, or, equivalently, that we have a very ample $\bbG_m$-linearized line bundle over $X$. 
Assume also that the $\bbG_m$-fixed point locus $X^{\bbG_m}$ is contained
in the regular locus of $X$. Define the \textit{attracting variety} $X'$  as follows~:
$$X'=\{x \in X\,;\, \lim_{t \to 0} t\cdot x\;\text{exists}\}.$$
 Then we have a partition $X'=\bigsqcup_{\rho}X_{\rho}$ into locally closed subsets with affine space bundles 
$p_{\rho}$ as above such that $X^{\bbG_m}=\bigsqcup_\rho F_\rho$ is the decomposition into connected components.
Further, there is a filtration $(X_{\leqslant \rho})$ of $X$ by closed subsets and an ordering of the components $F_\rho$ such that 
$X_\rho=X_{\leqslant\rho}\setminus X_{<\rho}$. Note that this filtration depends on the choice of the equivariant embedding of $X$ 
in a projective space with a diagonalizable $\bbG_m$-action.

\smallskip

\subsubsection{Roots} Let $T$ be the $n$-dimensional torus $(\bbG_m)^n$.
Let 
$$X_*(T)= \Hom( \mathbb{G}_m, T),\quad X^*(T)= \Hom( T,\mathbb{G}_m)$$ 
be the groups of cocharacters and characters of $T$.
We identify both with $\Z^n$ in the obvious way.

\smallskip

The set of \emph{roots} of a smooth $T$-variety $X$  is the set of characters of $T$ appearing in the normal bundles
to the connected components of the fixed points locus $X^T$. We'll say that a cocharacter $\gamma\in X_*(T)$
is \emph{generic} if the canonical pairing $\alpha\cdot\gamma$ is nonzero for each root $\alpha\in X^*(T)$.
In this case, the fixed point locus $X^\gamma=X^{\gamma(\bbG_m)}$ coincides with $X^T$.

\medskip

\subsection{The Bialynicki-Birula decompositions of $\frakM(\v,\w)$}\hfill\\

The subvarieties $\frakL^\flat(\v,\w),$  $\frakL(\v,\w)$ and $\frakM^{\flat}(\v,\w)$ of $\frakM(\v,\w)$ 
all have a geometric origin given in terms of attracting 
subvarieties for suitable $\mathbb{G}_m$-actions on $\frakM(\v,\w)$. Let us explain this. 

\smallskip

\subsubsection{The cases
of $\frakL^0(\v,\w)$, $\frakL(\v,\w)$ and $\frakM^0(\v,\w)$.} Let $T=(\mathbb{G}_m)^2$ be the two-dimensional torus acting on the $\kk$-variety $M(\v,\w)$ by
$$(t_1,t_2) \cdot (\bar x, p, q)=(t_1x, t_2x^*, t_2p, t_1q).$$
This yields a $T$-action on $\frakM(\v,\w)$, $\frakM_0(\v,\w)$ such that the map $\pi$ is $T$-equivariant and
$$\frakM^0(\v,\w)^{T}=\{0\}.$$
We deduce that the fixed point variety $\frakM(\v,\w)^{T}$ decomposes as 
a disjoint union
$$\frakM(\v,\w)^T = \bigsqcup_{\rho \in \chi} F_{\rho}$$
of smooth connected projective varieties $F_{\rho}$ which embed into $\frakL(\v,\w)$.

\smallskip

Now, given a cocharacter $\gamma$, the variety $\frakM(\v,\w)^{\gamma}$
is smooth, not necessarily projective and it is
the disjoint union of smooth connected closed subvarieties
$$\frakM(\v,\w)^{\gamma}=\bigsqcup_{\kappa \in \chi_{\gamma}} F_{\kappa}^\gamma.$$
Let $\frakM_\gamma \subseteq \frakM(\v,\w)$ be the attracting subvariety defined by
$$\frakM_\gamma=\{x \in \frakM(\v,\w)\,;\, \lim_{t \to 0} \gamma(t) \cdot x\;\text{exists}\}.$$

\begin{proposition}\label{P:attract} Let $\gamma=(a,b)$ and fix $\v,$ $\w$ in $\bbN^I$.
If the field $\kk$ is large enough then 
\hfill
\begin{itemize}[leftmargin=8mm]
\item[$\mathrm{(a)}$] $a,b\geqslant 0\Rightarrow\frakM_\gamma=\frakM(\v,\w)$,

\item[$\mathrm{(b)}$] $a,b<0\Rightarrow\frakM_\gamma=\frakL(\v,\w)$,

\item[$\mathrm{(c)}$] \begin{enumerate}[leftmargin=0mm]
\item[] $a<0=b\Rightarrow\frakM_\gamma=\frakL^0(\v,\w)$,
\item[] $b<0=a\Rightarrow\frakM_\gamma=\frakL^{*,0}(\v,\w)$,
\end{enumerate}
\item[$\mathrm{(d)}$] \begin{enumerate}[leftmargin=0mm]
\item[] $a <0<b $ with $|b/a| \gg 1\Rightarrow\frakM_\gamma=\frakM^0(\v,\w)$,
\item[] $b <0<a$ with $|a/b| \gg 1\Rightarrow\frakM_\gamma=\frakM^{*,0}(\v,\w)$.
\end{enumerate}
\end{itemize}
\end{proposition}

\smallskip

\begin{proof} Parts (a) and (b), are well-known, see \cite{Nak}.
Part (c) is stated in \cite[\S 2.1]{Bozec}. Part (d) is proved using a similar argument. For the convenience of the reader, we give a proof of both here. Since we may assume that the field $\kk$ is as large as we want, we may as well assume that it is indeed algebraically closed.
In addition, it is clearly enough to treat the first statements of (c) and (d).

\smallskip

We have
$$\frakL^0(\v,\w)=\pi^{-1}\pi(\frakL^0(\v,\w)), \quad \frakM^0(\v,\w)=\pi^{-1}\pi(\frakM^0(\v,\w)).$$
To see this, we may for instance use the description of the projective map $\pi : \frakM(\v,\w) \to \frakM_0(\v,\w)$
as a semi-simplification map, see \cite[prop. 3.20]{Nakajima} together with the fact that a filtered representation 
$\overline{x} \in \overline{E}_{\v}$ is semi-nilpotent if and only its associated graded is (resp. a filtered representation $x \in E_{\v}$ is nilpotent if and only if its associated graded is).

\smallskip

Since the map $\pi$ is projective and $T$-equivariant, for any value of $\gamma$ we have
$$\frakM_\gamma=\pi^{-1}\pi(\frakM_\gamma),\quad
\pi(\frakM_\gamma)=\{x \in \frakM_0(\v,\w)\,;\, \lim_{t \to 0} \gamma(t) \cdot x\;\text{exists}\}.$$ 
Hence it suffices to prove that for $a <0=b$ we have
\begin{equation}\label{E:propattract0}
\pi(\frakL^{0}(\v,\w))=\{x \in \frakM_0(\v,\w)\,;\, \lim_{t \to 0} \gamma(t) \cdot x\;\text{exists}\}
\end{equation}
while for $a<0<b$ and $|b/a| \gg 1$ we have 
\begin{equation}\label{E:propattract}
\pi(\frakM^{0}(\v,\w))=\{x \in \frakM_0(\v,\w)\,;\, \lim_{t \to 0} \gamma(t) \cdot x\;\text{exists}\}
\end{equation}

\smallskip

For each path $\sigma$ in $\bar Q$ and each tuple $(\bar x,p,q)\in\mu^{-1}(0)$,
let $\bar x_\sigma$ be the composition of $\bar x$ along $\sigma$. We'll view it as an element of 
$$\bar x_\sigma\in\k[\mu^{-1}(0)]\otimes \End(V).$$
By \cite[thm. ~1.3]{Lusztigonquiver}, \cite{LP} if $\text{char}(\kk)=0$, 
and \cite{Donkin} along with \emph{Crawley-Boevey's trick} in \cite{CB01} if $\text{char}(\kk) =p > 0$, 
the algebra $\k[\frakM_0(\v,\w)]$ is generated by the following two families of functions~: 
\begin{itemize}
\item[$(A_\sigma)$ :] the coefficients of the characteristic polynomials of $\bar x_{\sigma}$ 
for each oriented cycles $\sigma$ in $\bar Q$,
\item[$(B_\sigma)$ :] the functions $\phi(p\,\bar x_{\sigma}q)$ for each linear form $\phi$ on $\End(W)$
and each path $\sigma$ in $\bar Q$. 
\end{itemize}

\smallskip

The closed subset $\pi(\frakM_\gamma)$ of $\frakM_0(\v,\w)$ is the intersection of $\pi(\frakM(v,w))$ with the zero set of all
the functions in $(A_\sigma)$ or $(B_\sigma)$ which are of negative weight under the $\gamma$-action. 
Our aim is to determine precisely these sets under the assumption (c) (resp. (d)) on $\gamma$.

\smallskip

Let us first consider the situation in which $\gamma=(a,b)$ is as in case (c). Let $ z=(\bar x,p,q) \in \mu^{-1}(0)^s$ such that $[z] \in \frakM_\gamma$. We first claim that $q=0$. Indeed, otherwise, by the stability condition we have $p\, \bar{x}_{\sigma} q \neq 0$ for some
path $\sigma$ in $\bar Q$. But this defines a regular function on $\frakM_0(\v,\w)$ which is of negative weight. Next, let us prove that
$\bar x \in \Lambda^0_{\v}$. By our choice of $\gamma$, the characteristic polynomial of $\bar x_{\sigma}$ for any cycle in $\bar Q$ containing at least one arrow in $Q$ must be zero, and thus $\bar x_{\sigma}$ is nilpotent for any such cycle $\sigma$. In particular, $x_{\sigma}$ is nilpotent for any cycle $\sigma$ in $Q$. Let $A$ denote the image of the path algebra $\kk \bar Q$ under the natural evaluation map $\sigma \mapsto \bar x_{\sigma}$, and let $A^+ \subset A$ denote the unital subalgebra generated by all the paths
containing an arrow from $Q$. By the above, $A^+$ is a finite-dimensional algebra consisting of nilpotent endomorphisms hence by Wedderburn's theorem, $A^+$ is nilpotent. But then the flag of $I$-graded subspaces $K^{\bullet}$ defined by
$K^l= \Im ((A^+)^l)$ satisfies $V=K^0 \supset K^1 \supset \cdots \supset K^n=\{0\}$ for some large enough $n$, and
$$x_h (K^l) \subset L^{l+1}, \quad x_{h^*}(K^l) \subset K^l.$$
Therefore $\bar x$ is indeed semi-nilpotent, and $z \in \mu^{-1}(0)^s \cap (\Lambda^0_\v \times \bigoplus_i \text{Hom}(V_i,W_i))$.
We have shown that $\frakM_{\gamma} \subseteq \frakL^0(\v,\w)$. To prove the reverse inclusion, it is enough to notice that if
$z \in \mu^{-1}(0)^s \cap (\Lambda^0_\v \times \bigoplus_i \text{Hom}(V_i,W_i))$ then by the same argument as above, all the regular
functions of negative weight on $\frakM_0(\v,\w)$ vanish on $\pi([z])$, and thus $[z] \in \frakM_{\gamma}$. 

\smallskip

The argument in case (d) is very similar. Since $a<0$, for any oriented cycle $\sigma$ in $Q$ the characteristic polynomial of the monomial
$x_\sigma$ has a negative weight. Hence, it vanishes on $\pi(\frakM_\gamma)$. 
As above, we deduce that if $[\bar x,p,q]$ is a $\kk$-point of $\pi(\frakM_\gamma)$ then $x_\sigma$ is nilpotent for all such $\sigma$, 
hence $x$ is nilpotent.
Therefore, we have
$$\pi(\frakM_\gamma)\subseteq\pi(\frakM^0(\v,\w)).$$

\smallskip

It remains to prove the reverse inclusion if $b$ is large enough. 
Let $m=(\bar x, p, q)$ such that $[ m]$ belongs to $\frakM^{0}(\v,\w)$, i.e., the representation $x$ is nilpotent. 
We must prove that any function of type $(A_\sigma)$ or $(B_\sigma)$
which is of negative weight under the $\gamma$-action 
vanishes at $m$. We will prove this for functions of type $(A_\sigma)$ and 
leave the other case to the reader. 

\smallskip

Let $\sigma$ be an oriented cycle in $\bar{Q}$.
The weight of $\bar x_\sigma$ under the $\gamma$-action is of the form
$$a\,|\sigma\cap Q|+b\,|\sigma\cap Q^*|.$$
Assume that it is negative. Assume also that $b>2N|a|$ for some large enough positive integer $N$.
Then, we have $$|\sigma\cap Q|>2N\,|\sigma\cap Q^*|.$$
If $\sigma\cap Q^*=\emptyset$, then $\sigma$ is a path in $Q$, hence the characteristic polynomial of 
$\bar x_\sigma$ vanishes because the representation $x$ is nilpotent.
If $\sigma\cap Q^*\neq\emptyset$, then there is at least $N$ consecutive arrows in $\sigma\cap Q$.
Hence $\bar x_\sigma$ vanishes because the representation $x$ is nilpotent.
\end{proof}

\medskip

\subsubsection{The cases of $\frakL^{1}(\v,\w)$ and $\frakM^{1}(\v,\w)$.} This is very similar to the above, using a different torus action. The torus
$T'=\{(t_1,t_2,t_3,t_4)\in \mathbb{G}_m^4\;;\; t_1t_2=t_3t_4\}$ linearly acts on $\bar{E}_{\v}$ as follows~:
$$(t_1,t_2,t_3,t_4) \cdot x_h=\begin{cases} t_1x_h & \text{if}\; h' \neq h'',\, h \in \Omega\\ t_2x_h & \text{if}\; h' \neq h'', \,h \in \Omega^*\\ t_3x_h & \text{if}\; h' = h'',\, h \in \Omega\\ 
t_4x_h & \text{if} \;h' = h'',\, h \in \Omega^*.\end{cases}$$ 
Consider the action on $\frakM(\v,\w)$ defined as
$$(t_1,t_2,t_3,t_4) \cdot (\bar{x}, p, q)=((t_1,t_2,t_3,t_4) \cdot \bar{x}\,,\,p\,,\,t_1t_2q)$$
The group of cocharacters of $T'$ is identified with $\{ (u_1,u_2,u_3,u_4) \in \Z^4\;;\; u_1+u_2=u_3+u_4\}$. For any cocharacter $\gamma$
let $\frakM_{\gamma}$ be the attracting variety of the corresponding $\mathbb{G}_m$-action on $\frak{M}(\v,\w)$.

\smallskip

The following can be proved by the same type of arguments as in Proposition~\ref{P:attract}.

\smallskip

\begin{proposition}\label{P:attract2} Fix $\v,\w\in\bbN^I$ and $\gamma=(u_1,u_2,u_3,u_4)\in X_*(T')$.
If the field $\kk$ is large enough, then 
\hfill
\begin{itemize}[leftmargin=8mm]
\item[$\mathrm{(a)}$] $u_1,u_2,u_3<0, u_4=0\Rightarrow\frakM_{\gamma}=\frakL^{1}(\v,\w)$,
\item [$\mathrm{(b)}$] $u_1,u_2,u_3>0, u_4<0$ and $ |u_1/u_4|, |u_2/u_4|, |u_3/u_4| \gg 1 \Rightarrow\frakM_{\gamma}=\frakM^{*,1}(\v,\w)$.
\end{itemize}
\qed
\end{proposition}

\medskip

\subsection{More on the Bialynicki-Birula decompositions of $\frakM(\v,\w)$}\label{sec:3.3}\hfill\\

We now draw some consequences of the Bialynicki-Birula decompositions considered in the previous sections. Again, we
work out everything in details for $\frak{L}^0(\v,\w), $ $\frak{L}(\v,\w)$ and state the analogous results for $\frakL^{1}(\v,\w)$. 

\smallskip

\subsubsection{The cases of $\frakL^0(\v,\w),$ $ \frakL(\v,\w)$.} 
Fix $\v,$ $\w$ in $\N^I$ and fix a cocharacter $\gamma=(a,b)$  of $T$.
Assume that either the field $\kk$ is algebraically closed or that it is finite with a large enough number of elements.
Consider the $\bbG_m$-action on the $\kk$-variety $\frakM_\gamma$
associated with the cocharacter $\gamma$.
The Bialynicki-Birula decomposition gives a partition
$$\frakM_\gamma=\bigsqcup_{\kappa \in \chi_{\gamma}}\frakM_{\gamma,\kappa} $$
into locally closed $\kk$-subvarieties $\frakM_{\gamma,\kappa}$ and affine space bundles
$p_{\gamma,\kappa}~: \frakM_{\gamma,\kappa} \to F_{\gamma,\kappa}$ such that
$$\frakM(\v,\w)^\gamma=\bigsqcup_{\kappa \in \chi_{\gamma}}F_{\gamma,\kappa}$$ 
is the decomposition of the $\gamma$-fixed point set of $\frakM(\v,\w)$
into a disjoint union of smooth connected
closed subvarieties.

\smallskip

In order to relate the dimension of the affine space bundle 
$p_{\gamma,\kappa}$ for various $\gamma$ and $\kappa$, we consider the restriction of $p_{\gamma,\kappa}$ to the set of $T$-fixed points. 
Each $F_{\gamma,\kappa}$ is $T$-stable. We deduce that we have
$$\bigsqcup_{\rho \in \chi} F_{\rho}=\frakM(\v,\w)^T = \bigsqcup_{\kappa \in \chi_{\gamma}} (F_{\gamma,\kappa})^T.$$
This provides us with a map $\pi_{\gamma}:\chi \to \chi_{\gamma}$ such that we have
$$(F_{\gamma,\kappa})^T=\bigsqcup_{\rho \in \pi_{\gamma}^{-1}(\kappa)} F_{\rho}.$$
We claim that the map $\pi_\gamma$ is surjective. 
Indeed, by Proposition~\ref{P:attract}(a) there exists a generic cocharacter $\sigma \in X_*(T)$ 
which acts on $\frakM(\v,\w)$ in a contracting way. 
Therefore, every $T$-orbit in $\frakM(\v,\w)$ contains a $\sigma$-fixed point in its closure, and, by genericity of $\sigma$, 
a $\sigma$-fixed point is indeed fixed by $T$.
Since each $F_{\gamma,\kappa}$ is closed and $T$-invariant, it follows that each $F_{\gamma,\kappa}$ contains a $T$-fixed point.

\smallskip

Let $\Delta(\bfv,\bfw)$ be the set of roots of the $T$-action on $\frakM(\bfv,\bfw)$. It is a finite set.
Pick a point $z_{\rho} \in F_{\rho}$ for each $\rho$. Decompose the tangent space $ T_{\rho}$ of $\frakM(\v,\w)$ at $z_\rho$ 
as a sum of $T$-weight spaces.
We have
\begin{equation}\label{E:tcharac}
T_{\rho} = \bigoplus_{\alpha\in\Delta(\v,\w)}  T_{\rho} [\alpha],
\end{equation}
where the torus $T$ acts on $T_{\rho} [\alpha]$ via the character $\alpha$. 
The multiplicity $\dim(T_{\rho} [\alpha])$ does not depend on the choice of the element $z_\rho$ in $F_\rho$, because $F_\rho$ is 
connected.

\smallskip

Let $\Delta(\v,\w)_\gamma,$ $\Delta^+(\v,\w)_\gamma$ and $\Delta^-(\v,\w)_\gamma$ be the set of roots in $\Delta(\v,\w)$ given by
\begin{align*}
\Delta(\v,\w)_{\gamma}=\{\alpha \,;\, \gamma \cdot \alpha=0\}, \quad
\Delta^+(\v,\w)_{\gamma}=\{\alpha \,;\, \gamma \cdot \alpha>0\}, \quad
\Delta^-(\v,\w)_{\gamma}=\{\alpha \,;\, \gamma \cdot \alpha<0\}.
\end{align*}
For any $\kappa \in \chi_{\gamma}$ and $\rho \in \chi$ such that $\pi_{\gamma}(\rho)=\kappa$ we have
\begin{equation}\label{E:dimfix}
\dim(F_{\gamma,\kappa}) = \sum_{\alpha \in \Delta(\v,\w)_{\gamma}} \dim(T_{\rho} [\alpha]), \quad 
\dim( p_{\kappa,\gamma})= \sum_{\alpha \in \Delta^+(\v,\w)_{\gamma} }\dim(T_{\rho} [\alpha]).
\end{equation}

\smallskip

Since the $T$-action on $\frakM(\v,\w)$ scales the symplectic form $\omega$ by the character $\omega:=(1,1)$, 
the form $\omega$ restricts to a nondegenerate bilinear form $$\omega :T_{\rho}[\alpha] \times T_{\rho}[\omega-\alpha] \to \kk.$$
In particular, we have the following formula for each $\rho$ and $\alpha$
\begin{equation}\label{E:duality}
\dim(T_{\rho}[\alpha]) = \dim(T_{\rho}[\omega-\alpha]).
\end{equation}

\smallskip

 Recall that the cocharacter $\gamma$ is generic if the set $\Delta(\v,\w)_{\gamma}$ is empty.
 If $\gamma$ is generic, then we have 
 $\frakM(\v,\w)^{\gamma}=\frakM(\v,\w)^T$ and the map $\pi_{\gamma}$ is a bijection $\chi\to\chi_{\gamma}$
 such that $F_{\gamma,\pi_\gamma(\rho)}=F_{\rho}.$
 
\smallskip

Let us now examine the equation (\ref{E:dimfix}) in each of the cases occuring in Proposition~\ref{P:attract}.

\smallskip

\begin{itemize}
\item[(a)] Choose $a,b>0$ generic. Then for any $\rho \in \chi_{\gamma}=\chi$ we have
\begin{align}\label{E:case1}
\begin{split}
&\frakM_\gamma=\frakM(\v,\w),\quad
\frakM_{-\gamma}=\frakL(\v,\w),\\
\dim(p_{\gamma,\rho})& + \dim(p_{-\gamma, \rho}) + \dim(F_{\rho})= \dim(\frakM(\v,\w)).
\end{split}
\end{align}

\smallskip

\item[(b)] Choose $a<0=b$. Thus $\gamma$ may not be generic. Then, for any
$\kappa \in \chi_{\gamma}$ and $\rho \in \pi_{\gamma}^{-1}(\kappa)$
we have
\begin{equation}\label{E:case2}
\frakM_\gamma=\frakL^0(\v,\w),\quad \dim(p_{\gamma,\kappa})=\sum_{k<0}\sum_{l \in \Z} \dim(T_{\rho}[k,l]).
\end{equation}
Dually, if $b<0=a$, then for any
$\kappa \in \chi_{\gamma}$ and $\rho \in \pi_{\gamma}^{-1}(\kappa)$
we have
\begin{equation}\label{E:case2new}
\frakM_\gamma=\frakL^{*,0}(\v,\w),\quad
\dim(p_{\gamma,\kappa})=\sum_{k \in \Z}\sum_{l<0} \dim(T_{\rho}[k,l]).
\end{equation}
\smallskip

\item[(c)] Choose $a<0<b$ generic with $|b|/|a| \gg 1$. Then for any $\rho \in\chi_\gamma=\chi$ we have
\begin{equation}\label{E:case3}
\frakM_\gamma=\frakM^0(\v,\w),\quad
\dim(p_{\gamma,\rho})=\sum_{k\in \Z}\sum_{l >0} \dim(T_{\rho}[k,l]) + \sum_{k <0} \dim(T_{\rho}[k,0]) .
\end{equation}
Dually, if $a>0>b$ generic with $|a|/|b| \gg 1$
then we have
\begin{equation}\label{E:case3new}
\frakM_\gamma=\frakM^{*,0}(\v,\w),\quad
\dim(p_{\gamma,\rho})=\sum_{k >0}\sum_{l \in \Z} \dim(T_{\rho}[k,l]) + \sum_{l <0} \dim(T_{\rho}[0,l]) .
\end{equation}
\end{itemize}

\smallskip

By an \emph{almost affine space bundle} we'll mean a map which is a composition of affine space bundles.
We can now prove the following.

\smallskip

\begin{proposition}\label{P:fibrations}
\hfill
\begin{itemize}[leftmargin=8mm]
\item[$\mathrm{(a)}$] 
There exists partitions into locally closed $\kk$-subvarieties
$$\frakM(\v,\w)=\bigsqcup_{\rho \in \chi} Z_{\rho},\quad \frakL(\v,\w)= \bigsqcup_{\rho \in \chi} Y_{\rho}$$
and affine space bundles $u_{\rho} : Z_{\rho} \to F_{\rho}$ and $ v_{\rho}: Y_{\rho} \to F_{\rho}$ such that
\begin{equation}\label{E:eqdim1}
\dim(u_{\rho}) + \dim(v_{\rho}) +\dim(F_{\rho}) = \dim(\frakM(\v,\w)).
\end{equation}

\item[$\mathrm{(b)}$] There exists partitions into locally closed $\kk$-subvarieties
$$\frakM^{*,0}(\v,\w)=\bigsqcup_{\rho \in \chi} Z_{\rho},\quad \frakL^0(\v,\w)= \bigsqcup_{\rho \in \chi} Y_{\rho}$$
and almost affine space bundles $u_{\rho} : Z_{\rho} \to F_{\rho}$ and $ v_{\rho}: Y_{\rho} \to F_{\rho}$ such that
\begin{equation}\label{E:eqdim2}
\dim(u_{\rho}) + \dim(v_{\rho}) +\dim(F_{\rho}) = \dim(\frakM(\v,\w)).
\end{equation}
\end{itemize}
\end{proposition}

\smallskip

\begin{proof} To prove the claim (a), we choose $a,b$ as in \eqref{E:case1} and set
$$Z_{\rho}=\frakM_{\gamma,\,\rho},\quad u_{\rho}=
p_{\gamma,\,\rho},\quad Y_{\rho}=\frakM_{-\gamma,\,\rho},\quad v_{\rho}=p_{-\gamma,\,\rho}.$$

\smallskip

Let us turn to the claim (b). Given $a,b$ as in \eqref{E:case3}, we put
$$Z_{\rho}=\frakM_{\gamma,\,\rho},\quad u_{\rho}=p_{\gamma,\,\rho}.$$
To construct the partition of $\frakL^0(\v,\w)$ we proceed in two steps. For $\gamma=(-1,0)$ we consider the partition
$$\frakL^0(\v,\w)=\bigsqcup_{\kappa} \frakM_{\gamma,\kappa}$$ and the affine space bundles 
$p_{\gamma,\kappa}~: \frakM_{\gamma,\kappa} \to F_{\gamma,\kappa}$. 
Now pick some generic $a' , b'>0$. The $\gamma'$-action on $\frakM(\v,\w)$ is contracting. 
Since $F_{\gamma,\kappa}$ is smooth and 
$T$-stable, the Bialynicki-Birula decomposition gives us with a partition
$$F_{\gamma,\kappa}= \bigsqcup_{\rho \in \pi_{\gamma}^{-1}(\kappa)} U_{\rho,\kappa}$$
and affine space bundles $q_{\rho,\kappa} : U_{\rho,\kappa} \to F_{\rho}$ such that we have
\begin{equation}\label{E:dim4}
\dim(q_{\rho,\kappa})=\sum_{\alpha\in \Delta(\v,\w)_{\gamma} \cap \Delta^+(\v,\w)_{\gamma'}} \hspace{-.1in} 
\dim(T_{\rho}[\alpha]) = \sum_{l>0} \dim(T_{\rho}[0,l]).
\end{equation}

Now, for each $\rho\in\chi$ we set $\kappa=\pi_{\gamma}(\rho)$ and we define 
$$Y_{\rho}=p_{\gamma,\kappa}^{-1} (U_{\rho, \kappa}),\quad 
v_{\rho}=q_{\rho,\kappa} \circ p_{\gamma, \kappa}\,:\,Y_{\rho} \to F_{\rho}.$$
Using \eqref{E:case2}, \eqref{E:case3new} and \eqref{E:dim4} we compute
\begin{equation*}
\begin{split}
\dim(u_{\rho}) + \dim(v_{\rho}) + \dim(F_{\rho})
&= \dim(u_{\rho}) + \dim(p_{\gamma,\kappa}) + \dim(q_{\rho, \kappa}) + \dim(F_{\rho})\\
&=\sum_{k >0}\sum_{l \in \Z} \dim(T_{\rho}[k,l]) + \sum_{l <0} \dim(T_{\rho}[0,l]) + \sum_{k<0}\sum_{l \in \Z} \dim(T_{\rho}[k,l]) +\\
&\quad+\sum_{l >0} \dim(T_{\rho}[0,l]) + \dim(T_{\rho}[0,0])\\
&=\dim(T_{\rho})\\
&=\dim(\frakM(\v,\w))
\end{split}
\end{equation*}
as wanted.
\end{proof}

\medskip

\subsubsection{The cases of $\frakL^{1}(\v,\w)$.} Fix again some $\v,\w$ and assume that the
ground field $\kk$ is algebraically closed or that it is finite with a large enough number of elements.
The same arguments as above, applied to the decompositions
resulting from the $T'$-actions described in \S 3.3.2 yields the following result. Let us denote by
$$\frakM(\v,\w)^{T'}=\bigsqcup_{\rho \in \chi'} F_{\rho}$$
the decomposition of the $T'$-fixed point subvariety of $\frakM(\v,\w)$ into connected components. Each
$F_{\rho}$ is smooth and projective as $\frakM(\v,\w)^{T'} \subset \pi^{-1}(0)$.

\smallskip

\begin{proposition}
There exists partitions into locally closed $\kk$-subvarieties
$$\frakM^{*, 1}(\v,\w)=\bigsqcup_{\rho \in \chi'} Z_{\rho},\quad \frakL^{1}(\v,\w)= \bigsqcup_{\rho \in \chi'} Y_{\rho}$$
and affine space bundles $u_{\rho} : Z_{\rho} \to F_{\rho}$ and $ v_{\rho}: Y_{\rho} \to F_{\rho}$ such that
\begin{equation}\label{E:eqdim1new}
\dim(u_{\rho}) + \dim(v_{\rho}) +\dim(F_{\rho}) = \dim(\frakM(\v,\w)).
\end{equation}
\end{proposition}

\smallskip

\subsubsection{Irreducible components of $\frakL^0(\v,\w)$ and $\frakL^1(\v,\w)$.} The following simple corollary of (\ref{E:duality})
provides a geometric description and parametrization of the irreducible components of $\frakL^0(\v,\w)$ and $\frakL^1(\v,\w)$.

\smallskip

\begin{proposition}\label{P:lag}
\hfill
\begin{itemize}[leftmargin=8mm]
\item[$\mathrm{(a)}$] 
The variety $\frakL^0(\v,\w)$ is Lagrangian in $\frakM(\v,\w)$. Fix $\gamma=(a,b)$ with $a <0=b$. For any
$\kappa \in \chi_{\gamma}$, the smooth variety $\frakM_{\gamma,\kappa}$ is pure dimensional, and
$$\dim(\frakM_{\gamma,\kappa})=\dim(\frakL^0(\v,\w))=\frac{1}{2}\dim(\frakM(\v,\w)).$$
In particular, the irreducible components of $\frakL^0(\v,\w)$ are precisely the (Zariski closures) of the inverse images under the affine space bundles $p_{\gamma,\kappa}$ of the connected components of $F_{\gamma,\kappa}$ for $\kappa \in \chi_{\gamma}$.

\item[$\mathrm{(b)}$] 
The variety $\frakL^1(\v,\w)$ is Lagrangian in $\frakM(\v,\w)$ and is a closed subvariety in $\frakL^0(\v,\w)$.
In particular, the irreducible components of $\frakL^1(\v,\w)$ are the (Zariski closures) of the inverse images under the affine space bundles $p_{\gamma,\kappa}$ of certain smooth varieties $F_{\gamma,\kappa}$ for $\kappa \in \chi_{\gamma}$.
\end{itemize}
\end{proposition}

\smallskip

\begin{proof} The argument is similar to \cite[thm.~5.8]{Nakajima}. 
To prove (a) it is enough to show that for any $\kappa \in \chi_{\gamma}$ we have
$$\dim(p_{\gamma,\kappa}) + \dim(F_{\gamma,\kappa})=\frac{1}{2}
\dim (\frakM(\v,\w)).$$ 
Fix $\kappa \in \chi_{\gamma}$, and $z_{\rho} \in F_{\rho} \subset F_{\gamma,\kappa}$ for some $\rho \in \pi_{\gamma}^{-1}(\kappa)$ . By (\ref{E:case2}) we have
$$\dim (p_{\gamma,\kappa})=\sum_{k<0, l \in \Z} \dim(T_{\rho}[k,l]),\quad
\dim(F_{\gamma,\kappa})=\sum_{l\in \Z} \dim(T_{\rho}[0,l]).$$ But by (\ref{E:duality}) we have
$$\sum_{k\leq 0, l \in \Z} \dim(T_{\rho}[k,l]) = \sum_{k>0, l \in \Z} \dim(T_{\rho}[k,l])=\frac{1}{2} \sum_{k,l}\dim(T_{\rho}[k,l])=\frac{1}{2}\dim(\frakM(\v,\w))$$
from which we deduce the desired equality of dimension.  The proof that $\frakL^1(\v,\w)$ is Lagrangian can be found in \cite{BozecP}. The other statements follow from (a).
\end{proof}

\medskip

\subsection{Counting polynomials of quiver varieties}\hfill\\

In this paragraph we fix $\v,$ $\w$ and assume that the field $\fq$ is large enough to contain
the field of definition of the quiver varieties $\frakM(\v,\w),$ $ \frakL(\v,\w),$ etc, and 
the field of definition of the partitions and the affine space bundles considered in the previous section.

\smallskip

We will say that an $\fq$-algebraic variety $X$ has \emph{polynomial count} if there exists a polynomial $P(X,t)$ in $\Q[t]$ 
such that for any finite extension $\kk$ of $\mathbb{F}_{q}$ we 
have $|X(\kk)|=P(X,|\kk|).$

\smallskip

Let $\ell\neq p$ be a prime number. 
We will say that an $\fq$-algebraic variety $X$ is \textit{$\ell$-pure} (resp. \textit{very $\ell$-pure}) if all eigenvalues 
of the Frobenius endomorphims $F$ of the $\ell$-adic cohomology groups 
$H^i_c(X \otimes \fqb, \qlb)$ are of norm $q^{i/2}$ (resp. are equal to $q^{i/2}$). 
The Grothendieck-Lefschetz formula implies that
$$|X(\mathbb{F}_{q^r})|=\sum_i (-1)^i \Tr(F^r, H^i_c(X \otimes \fqb, \qlb)).$$
Hence if $X$ is very $\ell$-pure then it has polynomial count and its counting polynomial coincides with its $\ell$-adic Poincar\'e 
polynomial
 $$P_c(X,t):=\sum_i \dim\;H^{2i}_c(X \otimes \fqb,\qlb)\,t^i$$
Conversely, it is a classical result that if $X$ is of polynomial count and $\ell$-pure then it is very $\ell$-pure and has no odd 
cohomology  see, e.g., \cite[lem.~A.1]{CBVdB}.
Further, the counting polynomial $P(X,t)$ belongs to $\N[t]$.

\smallskip

The following result may be found in \cite[thm.~1]{Hausel}, see also \cite[prop.~6.1, thm.6.3]{Mozgovoy}. 

\smallskip

\begin{theorem}[Hausel]\label{P:ply0} The variety $\frakM(\v,\w)$ is very $\ell$-pure and of polynomial count. Its counting polynomial
(which is equal to its Poincar\'e polynomial in $\ell$-adic cohomology)
 is determined by the equality
 \begin{equation}\label{E:proof50}
\sum_{\v}
t^{-d(\v,\w)}P(\frakM(\v,\w),t)\,z^{\v}=\frac{r(\w,t,z)}{r(0,t,z)}.
\end{equation}
\qed
\end{theorem}

\smallskip

Our aim is now to generalize this result to the quiver varieties $\frakM^0(\v,\w),$ $\frakL(\v,\w),$ etc. 
First of all, the methods of Mozgovoy explained in \cite[\S 6]{Mozgovoy} can be adapted to get the following.

\smallskip

\begin{proposition}\label{P:ply0n} The varieties $\frakM^0(\v,\w)$, $\frakM^{1}(\v,\w)$ are very $\ell$-pure and of polynomial count. 
The counting polynomials of $\frakM^0(\v,\w)$ and $\frakM^{1}(\v,\w)$ 
(which are equal to their respective Poincar\'e polynomials in $\ell$-adic cohomology)
 are given by the following formula~:
\begin{equation}\label{E:county}
\begin{split}
&\sum_{\v}
t^{-d(\v,\w)}P(\frakM^0(\v,\w),t)\,z^{\v} =\frac{r^{0}(\w,t,z)}{r^{0}(0,t,z)},\\
&\sum_{\v}
t^{-d(\v,\w)}P(\frakM^{1}(\v,\w),t)\,z^{\v} =\frac{r^{1}(\w,t,z)}{r^{1}(0,t,z)}.
\end{split}
\end{equation}
The same holds for the varieties $\frakM^{*,0}(\v,\w)$ and $\frakM^{*,1}(\v,\w)$ respectively.
\end{proposition}

\smallskip

\begin{proof}
The proof given in~\cite[\S 6]{Mozgovoy} can be easily transposed to the present context, replacing Kac polynomials and $\frakM(\v,\w)$ by their nilpotent 
versions $A^{0}$ and $\frakM^{0}(\v,\w)$ or $A^{1}$ and $\frakM^{1}(\v,\w)$. We sketch the main steps for the reader's comfort,
in the case of $\frakM^0(\v,\w)$.

\smallskip

Fix $\v$ and $\w$. Consider the extended
quiver $\widetilde{Q}=(\widetilde{I}, \widetilde{\Omega})$ with vertex set $\widetilde{I}=I \cup \{\infty\}$ and with, in addition to the 
edges in $\Omega$,  $v_i$ oriented edges from $i$ to the new vertex $\infty$ for any $i \in I$.
Set $\tilde{\v}=(\v,1)$ in $\mathbb{N}^{\widetilde{I}}$ and write
$\widetilde{E}_{\tilde{\v}}=\bigoplus_{h \in \widetilde{\Omega}} \Hom(\k^{\tilde{v}_{h'}}, \k^{\tilde{v}_{h''}})$. 
Let 
\begin {align*}
\tilde{\mu} &: \widetilde{E}_{\tilde{\v}} \oplus \widetilde{E}_{\tilde{\v}}^* \to \frak{g}_{\tilde\v},\\
\mu &: M(\v,\w) \to \frak{g}_{\v}
\end{align*}
be the moment maps with respect to the actions of $G_{\tilde\v}$ and $G_\v$. 
Then, we have 
$$\mu^{-1}(0)^s \simeq \tilde{\mu}^{-1}(0)^s,$$ where the stability on the left hand side
is as in \S\,3.1 while the stability on the right hand side is 
defined in terms of some appropriate generic character. We deduce that
$$\frakM(\v,\w)=\widetilde{\mu}^{-1}(0)^s/\!\!/G_{\tilde\v}.$$ 
See Crawley-Boevey's trick in \cite{CB01}. 

\smallskip

Now, choose a 
generic $G_{\tilde\v}$-invariant line $L \subset \frakg_{\tilde{\v}}$ in such a way that we have
$$\widetilde{\mu}^{-1}(\xi)^s =\widetilde{\mu}^{-1}(\xi) \neq \emptyset,\quad\forall \xi \in L\setminus\{0\}.$$ 
Such a line exists if the field $\k$ is of large enough characteristic. 
Hence, setting $$\mathcal{X}=\widetilde{\mu}^{-1}(L)^s/\!\!/G_{\tilde\v},$$ we obtain a smooth variety equipped with a map 
$\pi: \mathcal{X} \to L$ such that $\pi^{-1}(0) = \frakM(\v,\w)$. 

\smallskip

Next we choose a
cocharacter $\gamma$ of $T$ as in Proposition~\ref{P:attract}(d), whose attracting variety in 
$\frakM(\v,\w)$ is $\frakM^0(\v,\w)$.
Consider the corresponding $\mathbb{G}_m$-action on $\mathcal{X}$ and $L$. Let 
$\mathcal{X}_{\gamma} \subset \mathcal{X}$ be the attracting variety. 

\smallskip

Arguing as in the proof of Proposition~\ref{P:attract}(d), we have 
$$\mathcal{X}_{\gamma}=(\widetilde{\mu}^{-1}(L)^s \cap (\widetilde{E}^0_{\tilde\v} \oplus \widetilde{E}_{\tilde\v}^*)) /\!\!/ G_{\tilde\v},$$
where
$\widetilde{E}^0_{\tilde\v}$ is the set of nilpotent representations of $\widetilde{Q}$. 
By the same argument as in \cite[Appendix]{CBVdB}, see also \cite[\S 3]{Mozgovoy}, 
the schemes 
$$\mathcal{X}_{\gamma,\,\xi} = \mathcal{X}_{\gamma} \cap \pi^{-1}(\xi),\quad \forall \xi \in L$$ 
all have the same class in the Grothendieck group of $\k$-schemes. So they have the same class as
$$\mathcal{X}_{\gamma,0}=\frakM^0(\v,\w).$$

\smallskip

Finally, the argument of \cite[prop.~2.2.1]{CBVdB} shows that for $\xi \neq 0$, the 
natural projection 
$$\mathcal{X}_{\gamma,\,\xi} \to \widetilde{E}^0_{\tilde{\v}}/ \!\!/G_{\tilde\v}$$ 
maps onto the set of geometrically 
indecomposable
nilpotent $\widetilde{Q}$-representations, and that we have
$$|\mathcal{X}_{\gamma,\,\xi}(\mathbb{F}_q) |=q^{d(\v,\w)}A^0_{\widetilde{Q}, \tilde\v}(q)$$
for large enough fields $\mathbb{F}_q$. This implies that $\frakM^0(\v,\w)$ also has polynomial count, with counting polynomial
equal to $q^{d(\v,\w)}A^0_{\widetilde{Q}, \tilde\v}(q)$. Formula (\ref{E:county}) is deduced from the above 
as in \cite[thm.~6.3]{Mozgovoy}.
\end{proof}

\smallskip

Next, we use the Bialynicki-Birula decompositions constructed in \S\ref{sec:3.3} to deal with the nilpotent quiver varieties.

\smallskip

\begin{proposition}\label{P:ply} The varieties $\frakL(\v,\w)$, $\frakL^\flat(\v,\w)$ and $\frakL^{*,\flat}(\v,\w)$ are very $\ell$-pure and of 
polynomial count over $\fq$.  
Their counting polynomials are equal to their $\ell$-adic Poincar\'e polynomials.
The counting polynomials satisfy the following relations~:
\begin{align}\label{E:proof6}
\sum_{\v}
t^{-d(\v,\w)}P(\frakL(\v,\w),t)\,z^{\v}&=\frac{r(\w,t^{-1},z)}{r(0,t^{-1},z)}\\
\sum_{\v}
t^{-d(\v,\w)}P(\frakL^\flat(\v,\w),t)\,z^{\v}&=\frac{r^{\flat}(\w,t^{-1},z)}{r^{\flat}(0,t^{-1},z)},\\
P(\frakL^\flat(\v,\w),t)&=P(\frakL^{*,\flat}(\v,\w),t).
\end{align}
\end{proposition}
\begin{proof} An analogous statement is proved in \cite[\S 5--8]{NakAnn} for virtual Hodge polynomials and
$ADE$ quivers. Our method is an adaptation of that proof. 
The proofs for $\frak{L}(\v,\w)$ and $\frakL^{\flat}(\v,\w)$ all proceed along the same lines. We will detail the proof for $\frakL(\v,\w)$ only.

\smallskip

We will first show that the varieties $F_{\rho}$, $\rho \in \chi$ are all very $\ell$-pure and have even cohomology.
The varieties $F_\rho$, being smooth and projective, they are pure. Hence the same holds for the affine space bundles
$u_{\rho}: {Z}_{\rho} \to F_{\rho}$ appearing in Proposition~\ref{P:fibrations}(a). We will now prove that each $F_\rho$ has polynomial 
count.

\smallskip

Since $\frakM(\v,\w)$ is quasi-projective, by Bialynicki-Birula there is an ordering $\leqslant$ on the set of connected components 
$\{F_\rho\,;\,\rho\in\chi\}$ such that 
the set $Z_{\leqslant \rho}:= \bigsqcup_{\eta \leqslant \rho} Z_{\eta}$ is open in $\frakM(\v,\w)$ for any $\rho$. 
Let 
$$i_{\rho} ~: Z_{<\rho} \to Z_{\leqslant \rho}, \quad j_{\rho} : Z_{\rho} \to Z_{\leqslant \rho}$$
be the obvious open and closed embeddings. Taking the hypercohomology of the exact triangle 
$$\xymatrix{ Ri_{\rho,!}\,i_{\rho}^! (\qlb)_{Z_{\leqslant \rho}} \ar[r]& (\qlb)_{Z_{\leqslant \rho}} \ar[r]& 
Rj_{\rho,*}\,j^*_{\rho} (\qlb)_{Z_{\leqslant \rho}}\ar^-{[1]}[r]&},$$
yields a long exact sequence in cohomology with compact support~: 
$$\xymatrix{
\cdots \ar^-{\delta_{i-1}}[r]& H^i_c( Z_{< \rho}, \qlb) \ar[r] & H^i_c(Z_{\leqslant \rho}, \qlb) \ar[r] & H^i_c(Z_{\rho}, \qlb) \ar^-{\delta_i}[r]& H^{i+1}_c(Z_{< \rho}, \qlb) \ar[r] &  \cdots}
$$
Arguing by induction, let us assume that $Z_{<\rho}$ is pure. Then the connecting homomorphisms $\delta_i$ are all zero and thus the long exact sequence above splits into short
exact sequences for each $i$, proving in turn that $Z_{\leqslant \rho}$ is pure. In particular, the variety $\frakM(\v,\w)$ is pure, see also \cite[prop. 6.2]{Mozgovoy}. 

\smallskip

The same argument proves in fact that
the set of Frobenius eigenvalues in $H^i_c(\frakM(\v,\w), \qlb)$, counted with multiplicity, is equal to the union of Frobenius eigenvalues in $H^{i-2d(\v,\w)}_c(F_\rho, \qlb)(d(\v,\w))$,
where $(d(\v,\w))$ is a Tate shift. By Proposition~\ref{P:ply0}, the variety $\frakM(\v,\w)$ has polynomial count. 
Hence it is very $\ell$-pure and we have $H^{2i+1}_c( \frakM(\v,\w), \qlb)=0$ for all $i$.  This implies that the variety $F_\rho$ satisfies the same properties for all $\rho$.

\smallskip

Reasoning in exactly the same fashion as above, but with the partition $\frakL(\v,\w) = \bigsqcup_{\rho} Y_{\rho}$ instead of $\frakM(\v,\w)=\bigsqcup_{\rho} Z_{\rho}$, we prove that 
the variety $\frakL(\v,\w)$ is also pure, of polynomial count. The same proof works for $\frakL^0(\v,\w)$ and $\frakL^{*,0}(\v,\w)$, using Proposition~\ref{P:fibrations}(b).

\smallskip

It now remains to compute the counting polynomials of $\frakL^0(\v,\w),$ $ \frakL^{*,0}(\v,\w)$ and $\frakL(\v,\w)$. To achieve this we will again use Proposition~\ref{P:fibrations} to relate the counting polynomials of $\frakL^0(\v,\w)$ and $\frakL(\v,\w)$ to that of $\frakM^{*,0}(\v,\w)$ and $\frakM(\v,\w)$ respectively. Using (\ref{E:eqdim1}) we compute
\begin{equation*}
\begin{split}
|\frakL(\v,\w)(\fq)|&=\sum_{\rho \in \chi} q^{\dim(v_{\rho})} |F_{\rho}(\fq)|\\
&= \sum_{\rho \in \chi} q^{\dim(\frakM(\v,\w))-\dim(u_{\rho}) - \dim(F_{\rho})} |F_{\rho}(\fq)|\\
&=q^{2d(\v,\w)}\sum_{\rho \in \chi}q^{-\dim(u_{\rho}) - \dim(F_{\rho})} |F_{\rho}(\fq)|.
\end{split}
\end{equation*}
Because $F_{\rho}$ is smooth and projective, we have by Poincar\'e duality
$$q^{-\dim(F_{\rho})}|F_{\rho}(\fq)|=q^{-\dim(F_{\rho})} P_c(F_{\rho}, q)=P_c(F_{\rho},q^{-1})$$
and thus
\begin{equation*}
\begin{split}
|\frakL(\v,\w)(\fq)|&=q^{2d(\v,\w)}\sum_{\rho \in \chi}q^{-\dim(u_{\rho})}P_c(F_{\rho},q^{-1})\\
&=q^{2d(\v,\w)}\sum_{\rho \in \chi}P_c(Z_{\rho},q^{-1})\\
&=q^{2d(\v,\w)} P_c(\frakM(\v,\w),q^{-1}).
\end{split}
\end{equation*}
Since this equality holds for any power of $q$, we deduce that 
$$P(\frakL(\v,\w),t)=t^{2d(\v,\w)}P(\frakM(\v,\w),t^{-1}).$$
Therefore for a fixed dimension vector $\w$ we have
$$\sum_{\v} t^{-d(\v,\w)}P(\frakL(\v,\w),t)\,z^\v=\sum_{\v} t^{d(\v,\w)} P(\frakM(\v,\w),t^{-1})\, z^{\v}=\frac{r(\w,t^{-1},z)}{r(0,t^{-1},z)}.$$
The determination of the counting polynomials of $\frakL^0(\v,\w)$ and $\frakL^{1}(\v,\w)$ proceed in the same fashion, using the 
affine space bundles in \eqref{E:eqdim2}, resp. the affine space bundles in (\ref{E:eqdim1new}), together with Proposition~\ref{P:ply0n}.
Finally, the equalities (\ref{E:proof6}) come from the fact that the Kac polynomials $A^0_{\v}$, $A^{1}_{\v}$, and hence the generating series $r^0(\w, t,z),$ $ r^{1}(\w,t,z)$ are invariant under reversing the orientation of all arrows in the quiver $Q$,
see Remark~\ref{R:kactranspo}.
\end{proof}

\medskip

\section{Proof of the theorem}

 \medskip
 
In this section we prove our Theorem~\ref{T:main}, by relating the volume of the stacks $\Lambda_{\v}$ and $\Lambda_{\v}^\flat$ to the 
number of points of nilpotent Nakajima quiver 
varieties $\frakL(\v,\w)$ and $\frakL^\flat(\v,\w)$. Again, the proofs of the three equalities in (\ref{E:theo}) being identical, we only deal 
with the one concerning 
$\Lambda_{\v}$.
  
\medskip

\subsection{The stratification of $\frakL(\v,\w)$}\hfill\\

We first relate the number of points of $\Lambda_{\v}$
and $\frakL(\v,\w)$
over finite fields.
Define a stratification of the variety $\frakL(\v,\w)$ by 
\begin{align*}
\frakL(\v,\w)&=\bigsqcup_{\w' \leqslant \w}\frakL(\v,\w)_{\w'},\\
\frakL(\v,\w)_{\w'}&=\{ G_{\v} \cdot ( \bar x, p,0)\in
\frakL(\v,\w)\,;\,\dim(\Im(\bigoplus_{i}p_i))=\w'\}.
\end{align*}

\smallskip

First, assume that $\v=\w=\w'$. Then the tuple $p=(p_h)$ can be viewed as an element of $G_\v$ and the map $(\bar x, p,0) \mapsto
(x')$ with $x'=px p^{-1}$ defines an isomorphism
$\frakL(\v,\v)_{\v} \simeq \Lambda_\v$.

\smallskip

Now, for any $\w'$ let $\Gr_{\w'}^{\w}$ be the Grassmannian of $I$-graded subspaces of $W$ of
dimension $\w'$.
The projection $\frakL(\v,\w)_{\w'} \to \Gr^{\w}_{\w'}$ is a fibration with
fiber $\frakL(\v,\w')_{\w'}$.
It follows that
\begin{equation}\label{E:proof1}
|\frakL(\v,\w)(\mathbb{F}_q)|=\sum_{\w' \leqslant \w} |\Gr_{\w'}^{\w}(\fq)| \cdot
|\frakL(\v,\w')_{\w'}(\fq)|.
\end{equation}
Inverting (\ref{E:proof1}) to express the number of $\fq$-points of
$\frakL(\v,\w)_{\w}$ for all $\w$ in terms
of the number of $\fq$-points of $\frakL(\v,\w)$ for all $\w$ yields, after a
small computation, the following.

\smallskip

\begin{lemma}\label{L:proof1} We have
\begin{align}\label{E:proof2}
|\frakL(\v,\w)_{\w}(\fq)|=\sum_{\w' \leqslant \w} (-1)^{|\w|-|\w'|}
q^{u(\w,\w')}|\Gr_{\w'}^{\w}(\fq)|\cdot |\frakL(\v,\w')(\fq)|,
\end{align}
where 
\begin{align}\label{u}u(\w,\w')=\sum_i (w_i-w'_i)(w_i-w'_i-1)/2.\end{align}
\qed
\end{lemma}

\smallskip

Setting $\w=\v$ we deduce the following.

\smallskip

\begin{corollary}\label{C:proof2} We have
\begin{equation}\label{E:proof2.5}
|\Lambda_\v(\fq)|=\sum_{\w' \leqslant \v} (-1)^{|\v|-|\w'|}
q^{u(\v,\w')}|\Gr_{\w'}^{\v}(\fq)|\cdot |\frakL(\v,\w')(\fq)|.
\end{equation}
\qed
\end{corollary}

\medskip

\subsection{Proof of Theorem~\ref{T:main}}\hfill\\ 

We may now proceed to the proof of Theorem~\ref{T:main}.
It is a direct computation using (\ref{E:proof2.5}) together with
(\ref{E:proof6}). For this, we consider the formal series $T_\w(z)$ in $\mathbf{L}$ given by
\begin{equation}\label{E:proof11}
T_{\w}(z)=\sum_{\v}
\frac{|\frakL(\v,\w)_{\w}(\fq)|}{|G_{\w}(\fq)|}\,q^{\langle\v,\v\rangle}\,z^{\v}\in \frac{|\Lambda_{\w}(\fq)|}{|G_{\w}(\fq)|}\,q^{\langle\w,\w\rangle}\,z^{\w} + \bigoplus_{\v < \w} \C z^{\v}.
\end{equation}
Using Lemma~\ref{L:proof1} and Proposition~\ref{P:ply} we get
\begin{equation*}
\begin{split}
T_{\w}(z)&=\sum_{\w' \leqslant \w}
(-1)^{|\w|-|\w'|}\, q^{u(\w,\w')}\,\frac{|\Gr^\w_{\w'}(\fq)|}{|G_{\w}(\fq)|}\,
\sum_{\v} |\frakL(\v,\w')(\fq)|\,q^{\langle \v,\v\rangle}\,z^{\v}\\
&=\sum_{\w' \leqslant \w} (-1)^{|\w|-|\w'|}\,q^{u(\w,\w')}\,\frac{|\Gr^\w_{\w'}(\fq)|}{|G_{\w}(\fq)|}\,\frac{r(\w',q^{-1},q^{\w'}z)}{r(0,q^{-1},q^{\w'}z)},
\end{split}
\end{equation*}
where we set $(q^{\mathbf{x}}z)^{\mathbf{y}}=q^{\mathbf{x} \cdot\mathbf{y}}z^{\mathbf{y}}$.
Expanding \eqref{r} in powers of $z$, we get
\begin{equation*}
\begin{split}
\frac{r(\w',q^{-1}, q^{\w'}z)}{r(0,q^{-1},q^{\w'}z)}&=\\
=\bigg\{ \sum_{\nu}X(\nu,q)&\,q^{\w'\cdot| \nu_{>1}|}\, z^{|\nu|}\bigg\}\cdot \bigg\{\sum_{l \geqslant 0}
\sum_{\nu^{(1)},\dots,\,\nu^{(l)}} (-1)^l\,q^{\w'\cdot\sum_{k=1}^l |\nu^{(k)}|} \,z^{\sum_{k=1}^l|\nu^{(k)}|}\,\prod_{k=1}^lX(\nu^{(k)},q)\bigg\},
\end{split}
\end{equation*}
where $\nu,\nu^{(1)},\dots,\nu^{(l)}$ run over the set of all $I$-partitions, with $\nu^{(1)},\dots,\nu^{(l)}\neq0$.
Expanding this last equation, substituting it in the expression for $T_\w$ above and pairing the summands corresponding to tuples of I-partitions
$(\nu=\emptyset, \nu^{(1)}, \ldots, \nu^{(l)})$ and $(\nu=\nu^{(1)}, \nu^{(2)}, \ldots, \nu^{(l)})$ we obtain
\begin{equation}\label{E:proof10}
T_{\w}(z)=\frac{1}{|G_{\w}(\fq)|} \sum_{l \geqslant 0}\sum_{\nu^{(1)},\dots,\,\nu^{(l)}}
(-1)^{l-1}\,z^{\sum_{k=1}^l|\nu^{(k)}|}\,K_{\w}^{(l)}(\nu^{(1)},\dots,\nu^{(l)})\,\prod_{k=1}^l X(\nu^{(k)},q),
\end{equation}
where $\nu^{(1)},\dots,\nu^{(l)}$ are $I$-partitions with $\nu^{(1)}, \ldots, \nu^{(l)} \neq 0$ and
\begin{equation*}
K_{\w}^{(l)}(\nu^{(1)},\dots,\nu^{(l)})=
\sum_{\w' \leqslant \w} (-1)^{|\w|-|\w'|}\, q^{u(\w,\w')}\,
\big(q^{\w'\cdot |\nu^{(1)}_{>1}|}-q^{\w'\cdot|\nu^{(1)}|}\big)\,q^{\w'\cdot\sum_{k=2}^l |\nu^{(k)}|}\, |\Gr^\w_{\w'}(\fq)|.
\end{equation*}
Replacing $\w'$ by $\w-\w'$ and setting $\boldsymbol{1}=(1, \ldots, 1)\in \N^I$ we may rewrite $K_{\w}^{(l)}(\nu^{(1)},\dots,\nu^{(l)})$ as 
\begin{equation*}
\begin{split}
K_{\w}^{(l)}(\nu^{(1)},\dots,\nu^{(l)})
&= q^{\w \cdot (|\nu^{(1)}_{>1}| + \sum_{k=2}^l |\nu^{(k)}|)}\,
\bigg(\sum_{\w'\leqslant \w} (-1)^{|\w'|}\, q^{\frac{1}{2}\w' \cdot(\w'-\boldsymbol{1})}
\,q^{-\w'\cdot (|\nu^{(1)}_{>1}|+\sum_{k=2}^l |\nu^{(k)}|)}\, |\Gr^\w_{\w'}(\fq)| \bigg) \\
&\quad \quad -q^{\w \cdot (\sum_{k} |\nu^{(k)}|)}\bigg(\sum_{\w' \leqslant \w}
(-1)^{|\w'|} \,q^{\frac{1}{2}\w' \cdot(\w'-\boldsymbol{1})}\,  q^{-\w'\cdot\sum_{k=1}^l |\nu^{(k)}|}\, |\Gr^\w_{\w'}(\fq)|\bigg).
\end{split}
\end{equation*}
Now we use the following identity : for any $w\in \mathbb{N}$ we have 
\begin{equation*}
\sum_{w'=0}^{w} (-1)^{w'} q^{w'(w'-1)/2-aw'} |\Gr_{w'}^{w}(\fq)|=
\begin{cases} 0 & \quad \text{if}\; a=0, 1, 
\ldots w-1 \\
(1-q^{-1})\cdots (1-q^{-w}) & \quad \text{if}\; a=w. \end{cases}
\end{equation*}
This relation is a direct consequence of the $q$-binomial formula
$$\sum_{w'=0}^w q^{w'(w'-1)/2} \begin{bmatrix} w \\w'\end{bmatrix}_q x^j= (1+x) \cdots (1+q^{w-1}x)$$
with $x=-q^{-a}$.
This implies that 
$$K_{\w}^{(l)}(\nu^{(1)},\dots,\nu^{(l)})=
\begin{cases}
0&\ \text{if}\ \sum_{k=1}^l |\nu^{(k)}| - \w \not\in \N^I,\\
-|G_{\w}(\fq)|&\ \text{if}\ \sum_{k=1}^l |\nu^{(k)}|=\w.
\end{cases}$$
Comparing the coefficients of $z^{\w}$ in (\ref{E:proof10}) and
(\ref{E:proof11}) we obtain the equality
$$\frac{|\Lambda_{\w}(\fq)|}{|G_{\w}(\fq)|}\,q^{\langle \w,\w\rangle}=
\sum_{l\geqslant 0}\sum_{\nu^{(1)},\dots,\,\nu^{(l)}}(-1)^{l}\prod_{k=1}^lX(\nu^{(k)},q)$$
where the sums runs over all $l$-tuples of $I$-partitions $\nu^{(1)}, \ldots, \nu^{(l)}$ with nonzero parts such that $\sum_{k=1}^l|\nu^{(k)}|=\w$.
Summing over all $\w$ we finally obtain
$$\sum_{\w}\frac{|\Lambda_{\w}(\fq)|}{|G_{\w}(\fq)|}\,q^{\langle\w,\w\rangle}\,z^{\w}
=\frac{1}{1+ \sum_{\nu \neq 0}X(\nu,q)\,z^{|\nu|}}=\frac{1}{r(0,q^{-1},z)}.$$
The theorem is now a consequence of Hua's formula
$$r(0,q^{-1},z)=\text{Exp} \bigg(  \frac{1}{q^{-1}-1}\sum_{\v}
A_{\v}(q^{-1})\,z^{\v}\bigg).$$
Obeserve that the above proof does not use any particular properties of $X(\nu,t)$, and hence is applicable verbatim to the nilpotent
variants $X^0, X^{1}$.
\qed

\medskip

\section{Factorization of $\lambda_Q(q,z)$ and the strata in $\Lambda_{\v}$.}

\medskip

In this short section, we slightly refine the point count of the Lusztig lagrangians by computing the number of points
of certain strata in these lagrangians defined by Lusztig, and relevant to representation theory. 
We let $Q$ be an arbitrary quiver. Let $Q_J$ be the full subquiver of $Q$ corresponding to a subset of vertices $J \subseteq I$.
All the varieties associated with $Q_J$ will be denoted with a superscript $J$.
There is an obvious inclusion $\bbN^J\subseteq\bbN^I$.
If $\v \in \N^J$ then we have $\Lambda^{0,J}_{\v} \simeq \Lambda^0_{\v}$ so that there is a factorization
\begin{equation}\label{E:crystal1}
\lambda^0_Q(q,z)=\lambda^0_{Q_J}(q,z) \cdot \lambda^0_{Q \backslash Q_J}(q,z)
\end{equation}
where we set
$$\lambda^0_{Q \backslash Q_J}(q,z)=\Exp\bigg( \frac{1}{1-q^{-1}} \sum_{\v\in\bbN^I\setminus\bbN^J} {A}^{0}_{\v}(q^{-1})\, z^{\v}\bigg).$$
The Fourier modes of $\lambda^0_{Q \backslash Q_J}(q,z)$ count the (orbifold) volume of some subvarieties in $\Lambda^0_{\v}$ considered by
Lusztig. These subvarieties are defined as follows. Let $K=\{k \in \bar\Omega\,;\, k' \in I \backslash J,\, k'' \in J\}$.
Given dimension vectors $\v\in\N^I$, $\mathbf{n}\in\N^J$ such that $\v-\mathbf{n} \in \N^I$, we set
$$\Lambda^0_{\v,\mathbf{n}} =\big\{ \bar x \in \Lambda^0_{\v}\,;\, \codim (\bigoplus_{k \in K} x_k)=\mathbf{n}\big\}.$$
Each $\Lambda^0_{\v,\mathbf{n}}$ is a locally closed subvariety in $\Lambda^0_{\v}$ and we have a stratification
$$\Lambda^0_{\v}=\bigsqcup_{\mathbf{n}} \Lambda^0_{\v,\mathbf{n}}.$$
There is natural map of stacks
$$p_{\v,\mathbf{n}}~:[\Lambda^0_{\v,\mathbf{n}}/G_{\v}] \to
[\Lambda^0_{\v,0}/G_{\v}] \times [\Lambda^{0,J}_{\mathbf{n}}/G_{\mathbf{n}}]$$
given by assigning to a representation $M=\bar x$ in $\Lambda^0_{\v,\mathbf{n}}$ the pair $(F, M/F)$ where $F$
is the subrepresentation of $M$ (for the doubled quiver $\bar{Q}$) generated by $\bigoplus_{i \not\in J} V_i$. The following
is proved in \cite[\S 12]{LusJAMS}. The proof there is given in the case where $J$ is reduced to a single vertex, but it is the same in
general.

\smallskip

\begin{proposition}\label{P:crystal} The map $p_{\v,\mathbf{n}}$ is a stack vector bundle of dimension $(\v-\mathbf{n},\mathbf{n})$.
\qed
\end{proposition}

\smallskip

Set 
$$\lambda^0_{Q,0}(q,z)=\sum_{\v} \frac{|\Lambda^0_{\v,0}(\fq)|}{|G_{\v}(\fq)|}\, q^{\langle \v,\v\rangle}\, z^{\v}.$$
From Proposition~\ref{P:crystal} we deduce the identity
\begin{equation}\label{E:crystal2}
\lambda^0_Q(q,z)=\lambda^0_{Q_J}(q,z) \cdot \lambda^0_{Q,0}(q,z).
\end{equation}

\smallskip

\begin{corollary} If the field $\kk$ is large enough, then we have
$\lambda^0_{Q,0}(q,z)=\lambda^0_{Q \backslash Q_J}(q,z).$
\qed
\end{corollary}

\smallskip

\begin{remark}The above corollary holds (with the same proof) for the two variants $\Lambda_{\v},$ $ \Lambda^{1}_{\v}$ of the Lusztig lagrangian.
\end{remark}

\medskip

\section{Appendix}

\smallskip

In this appendix, we prove an analogue, in the context of quivers in which $1$-cycles are allowed, of the Kac conjecture relating the constant terms of Kac polynomials to multiplicities of weights in the associated Kac-Moody algebra, see Theorem~\ref{T:Kac2}.

\smallskip

The theorem of Kashiwara and Saito relating the crystal of a Kac-Moody algebra to the set of irreducible components of Lusztig Lagrangians (\cite{KS}) has been generalized to an arbitrary quiver in \cite{Bozec}, where it is shown that the number of irreducible 
components of $\Lambda^{1}_{\v}$ is equal to the dimension of the $\v$ weight space in the (positive half) of the envelopping algebra of a certain explicit infinite-dimensional Lie algebra 
$\mathfrak{g}_Q$ attached to $Q$ which contains the Borcherds algebra attached to the adjacency matrix of $Q$. 
We'll prove that
$$a_{\v,0}^{1}=\dim(\mathfrak{g}_Q[\v]),$$
where $A_{\v}^{1}(t)=\sum_k a_{\v,k}^{1}t^k$. For instance, if $Q$ is the Jordan quiver then $\mathfrak{g}_Q$ is the Heisenberg 
algebra $\mathcal{H}=\C[t,t^{-1}] \oplus \C c$ and we have $\dim(\mathcal{H}[v])=1=a_{v,0}^{1}$ for any $v >0$.

\medskip

\subsection{The generalized quantum group}\hfill\\ 

We begin with some recollections of the generalized quantum group $U_Q$ 
associated to the quiver $Q$ defined in~\cite{BozecP}. 
We will use the notations and results of~\cite{BozecP,Bozec}. 
The symmetrized Euler form will still be denoted by $(\bullet,\bullet)$. 
We denote for simplicity by $\{i\,;\,i \in I\}$ the tautological basis of $\Z^I$. 
We let $I^\text{iso}$ be the set of isotropic vertices, i.e., the vertices satisfying $(i,i)=0$ (which means that there is exactly one loop at $i$ in the quiver $Q$). 
We also put 
$$I_\infty=(I^\text{re}\times\{1\})\sqcup (I^\text{im}\times\mathbb{N}_{>0})$$ and we extend the Euler form to $I_{\infty}$ by setting $((i,l),(j,k))=lk(i,j)$.

\smallskip{}

The $\Q(v)$-algebra $U_Q$ is generated by $\{E_\iota,F_\iota\,;\,\iota\in I_\infty\}$ and $\{K^{\pm1}_i\,;\,i\in I\}$, 
with respective degrees $li$, $-li$ and $0$ if $\iota=(i,l)$. These generators are subject to the following relations:\begin{align*}
K_iK_j&=K_jK_i,\\
K_iK_i^{-1}&=1,\\
K_jE_\iota&=v^{(j,\iota)}E_\iota K_j,\\
K_jF_\iota&=v^{-(j,\iota)}F_\iota K_j,\\
\sum_{t+t'=-(\iota,j)+1}(-1)^tE_j^{(t)}E_{\iota}E_j^{(t')}&=0&&\forall j\in I^\text{re},\\
\sum_{t+t'=-(\iota,j)+1}(-1)^tF_j^{(t)}F_{\iota}F_j^{(t')}&=0&&\forall j\in I^\text{re},\\
[E_\iota,E_{\iota'}]=[F_\iota,F_{\iota'}]&=0&&\text{if }(\iota,\iota')=0,
\end{align*}
along with some other relations coming from the Drinfeld double construction, which are not important for our purposes. We will also use an alternative set of primitive generators $a_{i,l}$ 
and $b_{i,l}$, also defined in \textit{loc. cit.}, of respective degree $li$ and $-li$. They satisfy a simpler set of relations, including
\begin{align*}
[a_{i,l},b_{i,l}]=\tau_{i,l}(K_{-li}-K_{li})
\end{align*}
for some constants $\tau_{i,l}\in\QQ(v)$.

\medskip

\subsection{Character formulas}\hfill\\

We now prove some character formulas, both for the algebra $U_Q$ and for its irreducible highest weight representations. 
Let $W$ be the Weyl group associated to $I^\textup{re}$.
For $\lambda\in P^+$, the set of dominant integral weights, let $\sigma_\lambda$ be the set of possible values for sums
\begin{align*}s=-\sum_{1\leqslant k\leqslant r}l_ki_k\end{align*} where $l_k>0$ and the vertices 
$i_k$ are pairwise orthogonal imaginary vertices, each perpendicular to $\lambda$. Note that this implies that if $i\notin I^\text{iso}$, then we have $|\{k\,;\, i_k=i\}|=1$. 
For such a sum, set
\begin{align*}
 \epsilon(s)=(-1)^\text{niso}\prod_{i\in I^\text{iso}}\phi_{\sum_{k\,;\,i_k=i}l_{k}}\end{align*}
 where
 \begin{align*} \text{niso}=|\{k\,;\, i_k\notin I^\text{iso}\}|\end{align*}
and $\phi(q)$ is the Euler function given by
\begin{align*} 
\phi(q)=\prod_{p\geqslant 1}(1-q^p)=\sum_{l\geqslant 0}\phi_lq^l=\sum_{n\in \ZZ}(-1)^nq^{(3n^2-n)/2}.\end{align*}

\smallskip

 Fix some formal variables $e^\alpha$ with $\alpha\in\ZZ I$ such that
 $e^{\alpha+\beta}=e^{\alpha}e^{\beta}$ for every $\alpha,\beta\in\ZZ I$ and $w.e^\alpha=e^{w(\alpha)}$ for every $w\in W$. Then we set
 \begin{align*} 
 S_\lambda=\sum_{s\in\sigma_\lambda}\epsilon(s)e^{s}.
 \end{align*}
The character of a $U_Q$-module $M$ is defined as \begin{align*}
\operatorname{Ch}(M)=\sum_{\lambda\in P}\dim (M_\lambda)\, e^\lambda\end{align*}
where $M_\lambda$ is the subspace of weight $\lambda$.

\smallskip

\begin{theorem} For every $\lambda\in P^+$, we have
\begin{equation}\label{carlambda}
\operatorname{Ch}(V(\lambda))=
\Big\{\sum_{w\in W}\epsilon(w)e^{-\rho+w(\lambda+\rho)}w(S_\lambda)\Big\}\operatorname{Ch}(U_Q^-),
\end{equation}
where $V(\lambda)$ is the simple module of highest weight $\lambda$,
and
\begin{equation}\label{carU}
\operatorname{Ch}(U_Q^-)=\Big\{\sum_{w\in W}\epsilon(w)e^{-\rho+w\rho}w(S_0)\Big\}^{\!\!-1}.\end{equation}
\end{theorem}

\smallskip

\begin{proof} 
The computation is close to the one made in the proof of~\cite[prop. 14]{BozecP}.
Recall that there exists a \emph{Casimir operator} $C$ acting on generalized Verma modules, defined in~\cite{BozecP}, and satisfying the following
relations~:
 \begin{align*}
K_i\,C&=C\, K_i,\\
K_{-li}\,a_{i,l}\,C&=K_{li}\,C\, a_{i,l},\\
b_{i,l}\,K_{li}\,C\, K_{li}&=C\, b_{i,l},
\end{align*}
for any $i\in I$ and $l\geqslant 1$.

\smallskip

Let $c$ be the $\QQ(v)$-linear map defined on the generalized Verma module $M(\lambda)$ of highest weight $\lambda$ by
\begin{align*}
c(m)=v^{f(\mu)}C m\text{ if }m\in V(\lambda)_\mu,\end{align*}
where $f(\mu)=(\mu,\mu+2\rho)$  and $\rho$ is defined by $(i,2\rho)=(i,i)$ for every $i\in I$. Notice that
\begin{align*}
f(\mu-li)-f(\mu)+2l(i,\mu) =l(l-1)(i,i)\end{align*}
for any $(i,l)\in I_\infty$. Since $C b_{i,l}=b_{i,l}C K_{2li}$, we get for any $m$
\begin{align*}
c(b_{i,l}m)&=v^{f(\mu-li)}C b_{i,l}m\\
&=v^{f(\mu-li)}b_{i,l}C K_{2li}m\\
&=v^{f(\mu-li)+2l(i,\mu)}b_{i,l}C m\\
&=v^{f(\mu-li)+2l(i,\mu)-f(\mu)}b_{i,l}c(m)\\
&=\left\{\begin{aligned}& v^{l(l-1)(i,i)}b_{i,l}c(m) &&\text{if }i\in I^\text{im}   \\ &b_{i,l}c(m) &&\text{if }i\in I^\text{re}  . \end{aligned}\right.
\end{align*}
Take $m\in M(\lambda)_\mu$ and assume $(\mu+\rho,i)\geqslant 0$ for every $i\in I$. Setting $\lambda-\mu=\alpha\in\mathbb{N} I$, we have
\begin{align*}
v^{f(\mu)}C m=v^{f(\lambda)+\sum_{1\leqslant k\leqslant r}l_k(l_k-1)(i_k,i_k)}C m\end{align*}
where $\sum_{i\in I^\text{im}}\alpha_ii=\sum_{1\leqslant k\leqslant r}l_ki_k$. If moreover $m$ is primitive, we get $c(m)=m$, hence
\begin{align*}
\displaystyle\sum_{1\le k\le r}l_k(l_k-1)(i_k,i_k)&=f(\mu)-f(\lambda)\\
&=(\mu-\lambda,\mu+\lambda+2\rho)\\
&=\underbrace{-\displaystyle\sum_{i\in I}\alpha_i(i,\lambda)}_{:=A\leqslant 0}-\displaystyle\sum_{i\in I}\alpha_i(i,\mu+2\rho)\\
&\leqslant -\displaystyle\sum_{i\in I^\text{re}}\alpha_i(i,\mu+2\rho)-\displaystyle\sum_{i\in I^\text{im}}\alpha_i(i,\mu+2\rho)\\
&\leqslant-\displaystyle\sum_{i\in I^\text{re}}{\alpha_i}-\displaystyle\sum_{i\in I^\text{im}}\alpha_i(i,\mu+2\rho)\\
&=-\displaystyle\sum_{i\in I^\text{re}}{\alpha_i}-\displaystyle\sum_{i\in I^\text{im}}\alpha_i(i,\lambda)+\displaystyle\sum_{i\in I^\text{im}}\alpha_i(i,\alpha-i)\\
&\leqslant-\displaystyle\sum_{i\in I^\text{re}}{\alpha_i}+\displaystyle\sum_{i\in I^\text{im}}\alpha_i(i,\alpha-i)\\
&=-\displaystyle\sum_{i\in I^\text{re}}\alpha_i+\sum_{i\in I^\text{im}}\alpha_i(\alpha_i-1)(i,i)+\displaystyle\sum_{\substack{i\in I^\text{im}\\j\neq i}}\alpha_i\alpha_j(i,j)\end{align*}
and thus
\begin{align*}
0\leqslant-\displaystyle\sum_{i\in I^\text{re}}\alpha_i+\displaystyle\sum_{\substack{i\in I^\text{im}\\j\neq i}}\alpha_i\alpha_j(i,j)
+\sum_{i\in I^\text{im}}(i,i)\bigg(\alpha_i(\alpha_i-1)-\displaystyle\sum_{i_k=i}l_k(l_k-1)\bigg).\end{align*}
Since $\sum_{i_k=i}l_k=\alpha_i$, we have
\begin{align*} 
\alpha_i(\alpha_i-1)-\sum_{i_k=i}l_k(l_k-1)\geqslant 0\end{align*}
with equality if and only if there is only one term in the sum.
Also, $(i,j)\leqslant 0$ when $i\neq j$, and $(i,i)\leqslant 0$ when $i$ is imaginary, hence
\begin{align}\label{ineq}
-\displaystyle\sum_{i\in I^\text{re}}\alpha_i+\displaystyle\sum_{\substack{i\in I^\text{im}\\j\neq i}}\alpha_i\alpha_j(i,j)
+\sum_{i\in I^\text{im}}(i,i)\bigg(\alpha_i(\alpha_i-1)-\displaystyle\sum_{i_k=i}l_k(l_k-1)\bigg)\leqslant 0.\end{align}
Finally every term in the sum is equal to $0$, hence $\alpha=\sum_{1\leqslant k\leqslant r}l_ki_k$, where the $i_k$ are pairwise orthogonal imaginary vertices, 
each perpendicular to $\lambda$ since $A$ has to be equal to $0$. This proves that if
\begin{align*}
\operatorname{Ch}(M(\lambda))=\displaystyle\sum_{\mu\leqslant \lambda}c'_\mu\operatorname{Ch}(V(\mu)),\end{align*}
where $c'_\lambda=1$, we have $c'_\mu=0$ if $\mu-\lambda\notin\sigma_\lambda$.

\smallskip

Now, as in~\cite{kacbook}, \cite{SVDB}, we also have
\begin{align*}
\dfrac{e^\rho\operatorname{Ch}(V(\lambda))}{\operatorname{Ch}(U_Q^-)}=
\displaystyle\sum_{\mu-\lambda\in\sigma_\lambda}c_\mu e^{\mu+\rho}\end{align*}
where $c_\mu\in \ZZ$ and $c_\lambda=1$, and both side are skew invariant under $W$. Also,
\begin{align*}
\dfrac{e^\rho\operatorname{Ch}(V(\lambda))}{\operatorname{Ch}(U_Q^-)}=
\displaystyle\sum_{w\in W}\epsilon(w)w(S_\lambda)\end{align*}
where
\begin{align*}
S_\lambda=\displaystyle\sum_{\substack{\mu-\lambda\in\sigma_\lambda\\\forall i,~(\mu+\rho,i)\geqslant 0}}c_\mu e^{\mu+\rho}.\end{align*}
If $\mu=\lambda-\alpha$ is a weight of $V(\lambda)$, there necessarily exists $i\in\operatorname{supp}(\alpha)$ such that $(i,\lambda)\neq0$. Indeed, if $(i,\lambda)=0$ and $l>0$, we have
\begin{align*}
a_{i,l}b_{i,l}v_\lambda=b_{i,l}a_{i,l}v_\lambda+\tau_{i,l}(K_{-li}-K_{li})v_\lambda=0+\tau_{i,l}(v^{-l(i,\lambda)}-v^{l(i,\lambda)})v_{\lambda}=0\end{align*}
and for any $\iota\neq(i,l)$
\begin{align*}
a_\iota b_{i,l}v_\lambda=b_{i,l}a_\iota v_\lambda.\end{align*}
Hence, by simplicity of $V(\lambda)$, we would have $b_{i,l}v_\lambda=0$. 
It proves that any term $e^{\rho+\mu}$ appearing in $S_\lambda$ comes from 
$e^{\rho+\lambda}/\operatorname{Ch}(U_Q^-)$, and we can conclude since for any 
$s=-\sum_{1\leqslant k\leqslant r}l_ki_k$ in $\sigma_\lambda$ the generators $F_{i_k,l_k}$ commute pairwise.
\end{proof}

\medskip

\subsection{The generalized Kac-Moody algebra}\hfill\\ 

We now define a generalized Kac-Moody algebra.
The Hopf algebra $U_Q$ can be seen as the quantized enveloping algebra of 
the Lie algebra $\mathfrak g_Q$ defined by the following set of generators
\begin{align*}
\{(h_i,e_{i,l},f_{i,l})\,;\,(i,l)\in I_\infty\}\end{align*}
subject to the following set of relations
\begin{align*}
[h_i,h_j]&=0\\
[h_j,e_\iota]&=(j,\iota)~e_\iota\\
[h_j,f_\iota]&=-(j,\iota)~f_\iota\\
(ad~e_j)^{-(j,\iota)}e_\iota=(ad~f_j)^{-(j,\iota)}f_\iota&=0&& \text{if }j\in I^\text{re}\\
[e_\iota,e_{\iota'}]=[f_\iota,f_{\iota'}]&=0&&\text{if }(\iota,\iota')=0\\
[e_\iota,f_{\iota'}]&=\delta_{\iota,\iota'}lh_i&&\text{if }\iota=(i,l).
\end{align*}
For $x\in\mathfrak g_Q$, we set $|x|=\alpha$ if $[h_i,x]=(i,\alpha)~x$ for every vertex $i$.
A generalized version of the Poincar\'e-Birkhoff-Witt theorem leads to the following formula~:\begin{equation}\label{pbw}
\operatorname{Ch}(U_Q^-)=\displaystyle\prod_{\alpha\in P^+}\dfrac{1}{(1-e^{-\alpha})^{\dim\mathfrak g_Q[\alpha]}}
\end{equation}
where $\mathfrak g_Q[\alpha]=\{x\in\mathfrak g_Q\,;\, |x|=\alpha\}$. Now, as in \S\ref{sec:1.7}, from the Lang-Weil theorem, we get
\begin{align*}
\frac{|\Lambda^{1}_{\alpha}(\mathbb{F}_q)|}{|G_{\alpha}(\mathbb{F}_q)|} q^{\langle \alpha, \alpha
\rangle}=|\text{Irr}(\Lambda^{1}_{\alpha})|
+O(q^{-1/2}).\end{align*}
But we know from~\cite{Bozec} that
\begin{align*}
|\text{Irr}(\Lambda^{1}_{\alpha})|=\dim(U_Q[\alpha])\end{align*}
hence
\begin{equation}\label{LWTH}
\lambda^{1}_{Q}(q,z)=\sum_{\alpha\in P^+} \dim(U_Q[\alpha])\, z^{\alpha} +
O(q^{-1/2}).
\end{equation}
On the other hand, by Theorem~\ref{T:main} we have
\begin{align}\label{ExpA}
\begin{split}
\lambda^{1}_{Q}(q,z)
&=\Exp\Big( \sum_{\alpha} A^{1}_{\alpha}(0)\,z^{\alpha} \Big) +
O(q^{-1/2})\\
&=\prod_{\alpha} 
(1-z^{\alpha})^{-A^\text{1}_{\alpha}(0)}+ O(q^{-1/2})
\end{split}
\end{align}
hence combining~\eqref{pbw},~\eqref{LWTH} and~\eqref{ExpA} we get the following result.

\smallskip

\begin{theorem}\label{T:Kac2} For any quiver $Q$ and any dimension vector $\alpha$ we have
\begin{equation*}
\dim(\mathfrak g_Q[\alpha])=A^{1}_{\alpha}(0).\end{equation*}
\qed
\end{theorem}

\smallskip

\begin{example} Let us consider again the case of the quiver with one vertex and $g$ loops. 

\smallskip

If $g=1$, we know that $A^{1}_\alpha(t)=A^{1}_{\alpha}(0)=1$ 
for all $\alpha\geqslant 1$. In this case, the Lie algebra $\mathfrak g_Q$ defined in this section
is the Heisenberg algebra generated by $e_l$, $f_l=e_{-l}$ for $l\geqslant 1$ and $h$ subject to the following relations~:\begin{align*}
[h,e_l]&=0\\
[e_l,e_k]&=l\delta_{l,-k}h.
\end{align*}
We see that $\mathcal H[\alpha]$ is just spanned by $e_\alpha$, hence of dimension $1$ as expected.

\smallskip

If $g>1$, the generators of $\mathfrak g_Q$ of positive (or negative) 
degree no longer commute, and $U_Q^+$ is the free 
noncommutative algebra spanned by one element at every positive degree. Hence
\begin{align*}
\operatorname{Ch}(U_Q^+)=1+\displaystyle\sum_{\alpha\geqslant 1}2^{\alpha-1}z^\alpha=\dfrac{1-z}{1-2z}=
(1-z)\displaystyle\prod_{\alpha\geqslant 1}(1-z^\alpha)^{-m(2,\alpha)}=(1-z)^{-1}\displaystyle\prod_{\alpha\geqslant 2}(1-z^\alpha)^{-m(2,\alpha)}
\end{align*}
if $m(k,\alpha)$ denotes the number of $k$-ary Lyndon words of length $n$, which is given by the Necklace polynomial
\begin{align*}
m(k,\alpha)=\dfrac{1}{\alpha}\displaystyle\sum_{d|\alpha}\mu\big(\dfrac{\alpha}{d}\big)k^d,\end{align*}
where $\mu$ is the M{\"o}bius function. This is consistent with the computation of $A^{1}_k$ for $k=1,2,3$ given at the end of \S \ref{sec:2.3} and with our definition of the free Lie algebra $\mathfrak g_Q$.
\end{example}

\smallskip

\begin{remark}\label{R:last} One may wonder about a similar Lie-theoretic interpretation of the constant terms $a_{\v,0}^0$ of the nilpotent Kac polynomials $A^0_{\v}(t)$ associated to 
a quiver $Q$. For an \textit{arbitrary} tuple $\underline{\v}$ of elements of $\N^I$ summing to $\v$, let
$\pi_{\underline{\v}}: \widetilde{\mathcal{F}l}_{\underline{\v}} \to E_{\v}$ be the proper map considered by Lusztig, see \cite[Section~1.5]{LusJAMS}. 
Let $\mathcal{Q}_{\v}$ be the category of all semi-simple complexes whose simple factors occur
in $\pi_{\underline{\v},!}(\qlb)$ for some $\underline{\v}$ and let $\mathcal{K}_{\v}$ be the graded Grothendieck group of $\mathcal{Q}_{\v}$. As in \cite{LusJAMS}, the space $\mathcal{K}=\bigoplus_{\v} \mathcal{K}_{\v}$ is equipped with a (twisted) Hopf algebra structure. It is natural to expect that $\mathcal{K}$ is a $q$-deformation of the positive half of the envelopping algebra of a certain Lie algebra $\mathfrak{g}'_Q$, and that $a^0_{\v,0}=\dim(\mathfrak{g}'_Q[\v])$. 

\smallskip

Note that the difference with Lusztig's original construction is that we
allow here \textit{arbitrary} tuples $\underline{\v}$, whereas Lusztig only considered \textit{restricted} tuples $\underline{\v}$, i.e., tuples
$\underline{\v}=(\v^{(k)})$ for which each $\v^{(k)}$ is concentrated at a single vertex $i \in I$.  By construction, we would have $\mathfrak{g}'_Q \supseteq \mathfrak{g}_Q$, 
but the two Lie algebras would differ if $Q$ contains some oriented cycle. For instance, if $Q$ is a cyclic quiver with $n$ vertices then one can show that 
$\mathfrak{g}_Q=\widehat{\mathfrak{sl}}_n$ while $\mathfrak{g}'_Q=\widehat{\mathfrak{gl}}_n$, see \cite{SIMRN}.
\end{remark}

\medskip

\centerline{\textbf{Acknowledgements}}

\vspace{.15in}

We would like to thank A. Chambert-Loir, T. Hausel, S.-J. Kang, E. Letellier, H. Nakajima and Y.~Soibelman for useful
discussions and correspondence.
As explained to one of us by Y. Soibelman, a formula similar to (\ref{E:theo}) in the case of quivers without 1-loops has been obtained independently in some joint work of his and Kontsevich.

\bigskip


\vspace{4mm}

\noindent

T. Bozec, \texttt{tbozec@mit.edu},\\
Department of Mathematics,
Massachussets Institute of Technology,
77 Massachusetts Avenue
Cambridge, MA 02139-4307, USA

\smallskip

O. Schiffmann, \texttt{olivier.schiffmann@math.u-psud.fr},\\
D\'epartement de Math\'ematiques, Universit\'e de Paris-Sud, B\^atiment 425
91405 Orsay Cedex, FRANCE.

\smallskip

E. Vasserot, \texttt{eric.vasserot@imj-prg.fr},\\
D\'epartement de Math\'ematiques, Universit\'e de Paris 7, B\^at Sophie Germain, 5 rue Thomas Mann. 75205 Paris CEDEX 13, FRANCE


\begin{thebibliography}{20000}


\bibitem{BozecP}
T. Bozec,
\newblock Quivers with loops and perverse sheaves.
\newblock {\em Math. Ann.}, \textbf{362}, (2015), 773-797.

\bibitem{Bozec}
T. Bozec, \emph{Quivers with loops and generalized crystals}, Compositio Mathematica, \textbf{152}(10), pp. 1999--2040 (2016).

\bibitem{Bridgeland}
T. Bridgeland, \emph{An introduction to motivic Hall algebras}, Adv. Math. \textbf{229}, no. 1, 102--138 (2012).

\bibitem{CBlectures}
W. Crawley-Boevey, \emph{Lectures on representations of quivers}, available at http://www1.maths.leeds.ac.uk/~pmtwc/

\bibitem{CB01} W. Crawley-Boevey,
\emph{Geometry of the moment map for representations of quivers,}
Compositio Math. 126 (2001), 257-293. 

\bibitem{CBVdB}
W. Crawley-Boevey, M. Van den Bergh, \emph{Absolutely indecomposable representations and Kac-Moody Lie algebras. 
With an appendix by Hiraku Nakajima}, Invent. Math. {155} (2004), 537-559. 

\bibitem{Donkin}
S. Donkin, \emph{Polynomial invariants of representations of quivers}, Comment. Math. Helv. {69} (1994), 137-141.

\bibitem{Feit}
W. Feit, N. J. Fine, \emph{Pairs of commuting matrices over a finite field},
Duke Math. J. {27} (1960) 91-94. 

\bibitem{GHS}
O. Garcia-Prada, J. Heinloth, A. Schmitt,
\emph{On the motives of moduli of chains and Higgs bundles}, J. Eur. Math. Soc. \textbf{16} (2014), 2617-2668.

\bibitem{Ginzburg}
V. Ginzburg, \emph{Lectures on Nakajima's quiver varieties}, S\'eminaires et Congr\`es {24-I} (2012), 143-197.

\bibitem{Hausel}
T. Hausel, \emph{Kac's conjecture from Nakajima quiver varieties}, Invent. Math. {181} (2010),  21-37. 

\bibitem{HRV}
T. Hausel, F. Rodriguez-Villegas, \emph{Mixed Hodge polynomials of character varieties. With an appendix by Nicholas M. Katz}, Invent. Math. {174} (2008),  555-624.

\bibitem{HLV}
T. Hausel, E. Lettellier, F. Villegas-Rodriguez, \emph{Positivity of Kac polynomials and DT-invariants for quivers},  Ann. of Math. {177} (2013), 1147-1168.

\bibitem{Hesselink}
W. Hesselink, \emph{Concentration under actions of algebraic groups }, 
S\'eminaire d'Alg\`ebre Paul Dubreil et Marie-Paule Malliavin,
Lecture Notes in Mathematics No. {867}, 1981, pp 55--89. 


\bibitem{Hua}
J. Hua, \emph{Counting representations of quivers over finite fields}, J. Algebra {226}, 1011-1033 (2000)

\bibitem{Kac}
V. Kac, \emph{Root systems, representations of quivers and invariant theory}, In: Invariant Theory, Montecatini, 1982. Lecture Notes in Mathematics, vol. {996}, 74-108, 
Springer, Berlin (1983). 


\bibitem{kacbook}
V. Kac,
\newblock {\em Infinite-dimensional {L}ie algebras}.
\newblock Cambridge University Press, Cambridge, third edition, 1990.


\bibitem{KS}
M. Kashiwara, T. Saito, \emph{Geometric construction of crystal bases}, Duke Math. J. {89} (1997), 9-36.

\bibitem{LP}
L. LeBruyn, C. Procesi, \emph{Semisimple representations of quivers}, Trans. Amer. Math. Soc. {317} (1990), 585-598.

\bibitem{LusJAMS}
G. Lusztig, \emph{Quivers, perverse sheaves, and quantized enveloping algebras}, J. Amer. Math. Soc. {4} (1991), 365-421. 

\bibitem{Lusconj}
G. Lusztig, \emph{Canonical bases arising from quantized enveloping algebras. II}, Common trends in mathematics and quantum field theories (Kyoto, 1990). Progr. Theoret. Phys. Suppl. No. {102} (1990), 175-201 (1991).

\bibitem{Lusztigonquiver}
G. Lusztig, \emph{On quiver varieties}, Adv. in Math., {136}, 141-182 (1998).

\bibitem{MO}
D. Maulik, A. Okounkov, \emph{Quantum groups and quantum cohomology}, preprint arXiv:1211.1287 (2012).

\bibitem{Mozgovoy}
S. Mozgovoy, \emph{Fermionic forms and quiver varieties}, arXiv:math/0610084 (2006).

\bibitem{Mozgovoy2}
S. Mozgovoy, \emph{ Motivic Donaldson-Thomas invariants and McKay correspondence},  arXiv:1107.6044 (2011).

\bibitem{SMoz}
S. Mozgovoy, O. Schiffmann, \emph{Counting Higgs bundles}, arXiv:1411.2101 (2014).

\bibitem{Nak}
H. Nakajima, \emph{Instantons on ALE spaces, quiver varieties and Kac-Moody algebras}, Duke Math. J. {76} (1994), 365-416. 

\bibitem{Nakajima}
H. Nakajima, \emph{Quiver varieties and Kac-Moody algebras}, Duke Math. J. {91}, 515-560 (1998). 

\bibitem{NakJAMS}
H. Nakajima, \emph{Quiver varieties and finite dimensional representations of quantum affine algebras},  
J. Amer. Math. Soc. {14} (2001), no. 1, 145--238.

\bibitem{NakAnn}
H. Nakajima, \emph{Quiver varieties and $t$-analogs of $q$-characters of quantum affine algebras}, Ann. of Math. (2) {160} (2004),  1057-1097.

 
\bibitem{R-V}
F. Villegas-Rodriguez, \emph{Counting colorings on varieties}, Publ. Mat. 2007, Proceedings of the Primeras Jornadas de Teor\`ia de N\'umeros, 209-220.

\bibitem{SIMRN}
O. Schiffmann, \emph{The Hall algebra of a cyclic quiver and canonical bases of Fock spaces}, IMRN, {8} (2000), 413--440.

\bibitem{SLectures2}
O. Schiffmann, \emph{Lectures on canonical and crystal bases of Hall algebras},  S\'eminaires et Congr\`es {24-II} (2012), 143-258.

\bibitem{SHiggs}
O. Schiffmann, \emph{Indecomposable vector bundles over smooth projective curves and stable Higgs bundles}, Ann. of Math. (2) {183} (2016), 297--362.

\bibitem{SV} 
O. Schiffmann, E. Vasserot, \emph{On cohomological Hall algebras of quivers}, in preparation. 

\bibitem{SVDB}
B. Sevenhant and M. Van~Den~Bergh,
\newblock A relation between a conjecture of {K}ac and the structure of the
  {H}all algebra.
\newblock {\em J. Pure Appl. Algebra}, 160, 319-332, 2001.

\bibitem{Wyss}
D. Wyss, \emph{Motivic classes of Nakajima quiver varieties}, Int. Math. Res. Notices (2016) doi: 10.1093/imrn/rnw217 


\end{thebibliography}
\end{document}